\documentclass{amsart}
\usepackage{amsmath}
\usepackage{amssymb}
\usepackage[all]{xy}
\usepackage{comment}
\usepackage{color}

\theoremstyle{plain}
\newtheorem{thm}{Theorem}[section]
\newtheorem{thm*}{Theorem}[section]

\newtheorem{prop}[thm]{Proposition}

\newtheorem{lemma*}{Lemma}

\newtheorem{question}[thm]{Question}

\theoremstyle{definition}
\newtheorem{defn}[thm]{Definition}
\newtheorem{remark}[thm]{Remark}

\newtheorem{note}[thm]{Notation}

\newtheorem*{remark*}{Remark}

\newtheorem{ex}[thm]{Example}

\numberwithin{equation}{thm}

\newcommand{\bR}{\mathbb R}

\newcommand{\bC}{\mathbb C}

\newcommand{\cN}{\mathcal N}

\def\Spec{\operatorname{Spec}\nolimits}

\def\Proj{\operatorname{Proj}\nolimits}

\newcommand{\bG}{\mathbb G}

\newcommand{\cO}{\mathcal O}
\newcommand{\bA}{\mathbb A}

\newcommand{\bZ}{\mathbb Z}
\newcommand{\bF}{\mathbb F}

\newcommand{\cC}{\mathcal C}

\newcommand{\cE}{\mathcal E}

\newcommand{\fg}{\mathfrak g}
\newcommand{\fh}{\mathfrak h}

\newcommand{\gl} {\mathfrak {gl}}

\newcommand{\ol}{\overline}
\newcommand{\ul}{\underline}

\def\Spec{\operatorname{Spec}\nolimits}
\def\sl2{\operatorname{SL_{2(2)}}\nolimits}
\def\Ga2{\operatorname{\mathbb G_{\rm a(2)}}\nolimits}

\newcommand{\Z}{\mathbb Z}
\newcommand{\bN}{\mathbb N}

\newcommand{\bu}{\bullet}

\date\today

\begin{document}

\title[Rational Cohomology and supports]{Rational cohomology and supports for Linear Algebraic Groups}
 
 \author[ Eric M. Friedlander]
{Eric M. Friedlander} 

\address {Department of Mathematics, University of Southern California,
Los Angeles, CA}
\email{ericmf@usc.edu}




\maketitle

\centerline{ Dedicated to David J. Benson}

\section{Introduction}

What follows are rough ``notes" based upon four lectures given by the 
author at PIMS in Vancouver over the period June 27 to June 30, 2016.
\begin{itemize}
\item
Lecture I.  Affine groups schemes over $k$.
\item
Lecture II.  Algebraic representations.
\item
Lecture III.  Cohomological support varieties.
\item
Lecture IV.  Support varieties for linear algebraic groups.
\end{itemize}
\vskip .1in

The primary goal of these lectures was to publicize the author's recent 
efforts to extend to representations of linear algebraic groups the 
``theory of support varieties" which has proved successful in the study
of representations of finite group schemes.  The first two lectures 
offer a quick review of relevant background for the study of affine group
schemes and their representations.  The third lecture discusses
cohomological support varieties of finite group schemes and mentions 
challenges to extending this theory to linear algebraic groups (also
discussed in the last paragraph of this introduction). 
In the fourth and final lecture, we provide an introduction to the author's theory of 
support varieties using 1-parameter subgroups following work of
A. Suslin, C. Bendel, and the author \cite{SFB1}, \cite{SFB2}.   The text
contains a few improvements on results in the literature (see, for example,
Remark \ref{rem:tw}).

We encourage others to follow the work discussed here by
sharpening the formulations, extending general theory, providing much
better computations, and working out many interesting examples.  
Towards the end of Lecture IV, we give a list of various explict problems
which might be of interest to some readers.
We conclude this text by introducing ``formal 1-parameter
subgroups" leading to ``formal support varieties" in Proposition \ref{prop:formal},
a promising but still unexplored structure.

The reader will find undue emphasis on the work of the author together 
with collaborators Chris Bendel, Jon Carlson, 
Brian Parshall, Julia Pevtsova, and Andrei Suslin.  A quick look at
references given will see that numerous other mathematicians have played 
seminal roles in developing support varieties, including Daniel Quillen 
who launched this entire subject.

We conclude this introduction with a brief sketch of the evolution of 
support varieties for a ``group-like object" $G$ and introduce our 
perspective on their role in the study of their representations (which we 
usually refer to as $G$-modules).  

Support varieties emerged from D. Quillen's work \cite{Q1}, \cite{Q2} on the cohomology 
of finite groups.   The reader attracted by homological 
computations might become distracted (as we have been, at times) from our goal of illuminating
representation theory by the numerous puzzles and questions concerning
cohomology which arise from the geometric perspective of support varieties.

In the late 1960's, and even now, few complete calculations of the cohomology
of finite groups were known.  Quillen developed foundations for the
equivariant cohomology theory introduced by A. Borel \cite{Borel}, a key tool in his 
determination of the prime ideal spectrum
of the cohomology algebra $H^*(G,k)$ for {\it any finite group} \cite{Q1},
\cite{Q2}.  This enabled Quillen to answer a question of Atiyah and Swan on 
the growth of projective resolutions of $k$ as a $kG$-module \cite{Swan}.  This is
just one example of Quillen's genius: proving a difficult conjecture by
creating a new context, establishing foundations, and proving a geometric,
refined result which implies the conjecture.  

A decade later, Jon Alperin and Len Evens considered the growth of 
projective resolutions of finite dimensional $kG$-modules \cite{AE1}, 
\cite{AE2}.  They recognized that Quillen's geometric description of the
``complexity" of the trivial module $k$ for $kG$ had an extension to 
arbitrary finite dimensional modules.  Following this, Jon Carlson formulated
in \cite{Ca1} the (cohomological) support variety $|G|_M$ of a finite dimensional 
$kG$-module, a closed subvariety of ``Quillen's variety" $|G|$.  At first
glance, one might think this construction is unhelpful:  one starts with
a $kG$-module $M$ and one obtains invariants of $M$ by considering the 
structure of the $Ext$-algebra $Ext^*_G(M,M)$ as a module over the 
cohomology algebra $H^*(G,k)$.  Yet in the hands of Carlson and others,
this has proved valuable in studying the representation theory of $G$
and more general ``group-like" structures.

One early development in the study of support varieties was an alternative
construction proposed by J. Carlson for an elementary 
abelian $p$-group $E \simeq (\bZ/p)^{\times s}$ and proved equivalent to the cohomological
construction by G. Avrunin and L. Scott \cite{AS}.  J. Carlson's fundamental
insight was to reformulate the cohomological variety
$|E|$ as a geometric object $V(E)$ whose points are related to $kE$ without
reference to cohomology; Carlson then reformulated the support variety $|E|_M$ 
of a finite dimensional $kE$-module in 
``local" terms without reference to homological constructions such as 
the Ext algebra $Ext_E^*(M,M)$.  Only much later was this extended by
J. Pevtsova and the author \cite{FPv1}, \cite{FPv2} to apply not just to 
elementary abelian $p$-groups but to all finite groups; indeed, in doing so, 
Friedlander and Pevtsova formulated this comparison for all finite group
schemes.

This leads us to other ``group-like" objects.  Such a consideration was
foreshadowed by the work of Avrunin-Scott  who solved Carlson's conjecture
by considering a different Hopf algebra (the restricted enveloping algebra
of an abelian Lie algebra with trivial restriction) whose underlying 
algebra is isomorphic to $kE$.   B. Parshall and the author wrote 
a series of papers (see, for example, \cite{FPar1},\cite{FPar2},\cite{FPar3},\cite{FPar4})
introducing and exploring a support theory for the $p$-restricted representations of 
an arbitrary finite dimensional restricted Lie algebra.
This entailed the consideration of the cohomology algebra $H^*(U^{[p]}(\fg),k)$ 
of the restricted enveloping algebra $U^{[p]}(\fg)$; modules for $U^{[p]}(\fg)$
are $p$-restricted representations of $\fg$.

Subject to restrictions on 
the prime $p$, work of Parshall and the author together with work of J. Jantzen \cite{J1}
showed that Carlson's conjecture for elementary abelian $p$-groups 
generalized to any finite dimensional restricted Lie algebra $\fg$, 
comparing $|\fg|_M$ to the generalization $V(\fg)_M$
of Carlson's rank variety defined 
in ``local" representation-theoretic terms rather than using homological 
constructions.  (Subsequently, A. Suslin, C. Bendel and the author formulated
and proved such a comparison for all primes $p$ \cite{SFB2}.)
This comparison enables proofs of properties for the
support variety construction $M \mapsto V(\fg)_M$, some of which  are
achieved using homological methods and some using more geometric, 
representation theoretic methods.

In this paper, ``linear algebraic group over $k$" refers to a reduced, 
irreducible, affine group scheme of finite type over $k$, always assumed to be
of characteristic $p > 0$ for some prime $p$.  For such a linear 
algebraic group $G$,  there is a Frobenius morphism $F: G \to G^{(1)}$; if $G$ is
defined over the prime field $\bF_p$, the Frobenius morphism is an endomorphism
$F: G \to G$.  The kernel of $F$ is a height 1 ``infinitesimal group scheme" of 
finite type over $k$ denoted $G_{(1)}$. 
For a finite group scheme of the form $G_{(r)} \ = \ ker\{ F^r:G \to G^{(r)} \}$
(the $r$-th Frobenius kernel of the linear algebraic group $G$) and a finite 
dimensional $G_{(r)}$-module $M$, A. Suslin, C. Bendel, and the author 
give in \cite{SFB1} and \cite{SFB2}  a representation-theoretic formulation, denoted $V(G_{(r)})_M$,
of  the cohomological support variety $|G_{(r)}|_M$.
This alternate description is formulated in terms of the restriction of $M$ along infinitesimal 1-parameter subgroups 
$\psi: \bG_{a(r)} \to G$.

Finite groups and Frobenius kernels of linear algebraic groups are examples
of finite group schemes.  In \cite{FPv1}, \cite{FPv2}, J. Pevtsova and the author extended the 
theory of support varieties to arbitrary finite group schemes, generalizing
``cyclic shifted subgroups" considered by J. Carlson in the case of
elementary abelian $p$-groups and reinterpreting infinitesimal
1-parameter subgroups considered by A. Suslin, C. Bendel, and the author
in the case of infinitesimal group schemes over $k$.
The finite dimensionality of the cohomology algebra  $H^*(G,k)$ of a finite group
scheme proved by A. Suslin and the author \cite{AS} plays a crucial role in
these theories of cohomological support varieties.  In these extensions
of the original theory for finite groups, one requires a suitable criterion of
the detection modulo nilpotence of elements of  $H^*(G,k)$; for finite groups,
such a detection result is one of D. Quillen's key theorems.

Although many of the basic techniques used in establishing properties
for cohomological support varieties for finite group schemes do not apply
to linear algebraic groups, we have continued to seek a suitable theory of
support varieties for linear algebraic groups.
 After all, a major justification for the consideration of
Frobenius kernels is that the collection $\{ G_{(r)}, r > 0 \}$ has 
representation theory that of $G$ whenever $G$ is a simply connected, simple
linear algebraic group as shown by J. Sullivan \cite{Sull} (see also \cite{CPS2}).  

However, the rational cohomology of a simple algebraic group
vanishes in positive degree, so that cohomological methods do not appear
possible.   Furthermore, if the rational cohomology is non-trivial, it is 
typically not finitely generated.  Finally, there are typically no non-trivial projective 
$G$-modules for a linear algebraic group as shown by S. Donkin \cite{Donk}.
Despite these difficulties, we present in Lecture IV a theory of 
support varieties for linear algebraic groups of ``exponential type".

Throughout these lectures, we use the simpler term ``$G$-module" rather than
than the usual ``rational $G$-module" when referring to a ``rational representation"
of an affine group scheme $G$.
We shall abbreviate $V \otimes_k W$ by $V\otimes W$ for the tensor product
of $k$-vector spaces $V, W$.


\vskip 1in


\section{Lecture I: Affine group schemes over $k$}

        This lecture is a ``recollection" of some elementary aspects
of affine algebraic varieties over a field $k$ and a discussion of
group schemes over $k$.   We recommend R. Hartshorne's book
``Algebraic Geometry" \cite{Har} and W. Waterhouse's book ``An introduction to
affine group schemes" \cite{Wat}  for further reading.  

We choose a prime $p$ and consider algebraic varieties over an algebraically
closed field $k$ of characteristic $p > 0$.  
The assumption that $k$ is algebraically closed both simplifies the algebraic 
geometry (through appeals to the Hilbert Nullstellensatz) and simplifies the 
form of various affine group schemes.  Our hypothesis that $k$ is not of 
characteristic 0 is necessary both for the existence of various finite group schemes
and for the non-triviality of various structures.  In Subsection I.3, we discuss some of the special
features of working over such a field rather than working over a 
field of characteristic 0.  In Subsection I.4, we discuss restricted Lie algebras
and their $p$-restricted representations. 

\vskip .1in

        Here is the outline provided to those attending this lecture.
        
        \vskip .1in
\noindent
I.A  Affine varieties over $k$.

        i.)    $k$ a field, alg closed;
        
        ii.)   $\bA^n$ -- affine space over $k$;
        
        iii.)  zero loci $Z(\{ f_1,\ldots,f_m \})  \subset \bA^n$;
        
        iv.)   Algebra $k[x_1,\ldots,x_n]/(f_1,\ldots,f_m)$ of algebraic
functions;

        v.)    Hilbert nullsellensatz:  $X \subset \bA^n$ versus
$k[x_1,\ldots,x_n] \to k[X]$.

        \vskip .1in
\noindent
I.B  Affine group schemes over $k$.

        i.)    Examples:  $\bG_a, \bG_m, GL_N, U_N$;
        
         ii.)   Product on $G$ gives coproduct on $k[G]$;
         
         iii.)  Group objects in the category of affine algebraic
varieties over $k$;

         iv.)   As representable functors from (comm k-alg) to (grps);
         
         v.)    Hopf algebras.

        \vskip .1in
\noindent
I.C  Characteristic $p > 0$.
         
         i.)    Examples of $k$ with $char(k) = p$;
        
         ii.)   Geometric Frobenius on affine varieties/$k$;
         
         iii.)  Lang map:  $1/F: G \to G$;
         
         iv.)   Frobenius kernels   $G_{(r)} = \ker\{ F^r: G \to G \}$;
        
         v.)    Example of $GL_{N(r)}$.

        \vskip .1in
\noindent
I.D  Lie algebra of $G$
         
         i.)    $Lie(G)$, tangent space at identity as derivations on $k[G]$.
        
         ii.)   Lie bracket $[-,-]$ and $p$-th power $(-)^{p]}$;
        
         iii.)  Examples of $\bG_a, \bG_m, GL_N$;
         
         iv.)   Relationship to $G_{(1)}$.

 \vskip .1in
\noindent
Supplementary topics:

\noindent
I.A  Extending consideration to $k$ not algebraically closed.
      Projective varieties.
\vskip  .1in
\noindent
I.B   Simple algebraic groups and their classification.
      Working with categories and functors.

\vskip  .1in
\noindent
I.C   Frobenius twists, $F: G \to G^{(1)}$.
      Arithmetic and absolute Frobenius maps.

\vskip  .1in
\noindent
I.D   Complex, simple Lie algebras.
      Root systems.
\vskip .2in

\subsection{Affine group schemes over $k$}

        We begin with a cursory introduction to affine algebraic geometry
over an algebraically closed field $k$.  For any $n > 0$, we denote by 
$\bA^n$ the set of $n$-tuples of of elements of $k$, by $\ul a \in \bA^n$
a typical $n$-tuple.  What distinguishes algebraic geometry from other
types of geometries is the role of algebraic functions on an algebraic
variety.  The ring of algebraic (i.e., polynomial) functions on $\bA^n$ is
by definition the $k$-algebra $k[x_1,\ldots,x_n]$ of polynomials in $n$ 
variables with coefficients in $k$, 
$p(\ul x) = \sum_{\ul d} c_{\ul d} \ul x^{\ul d}$, where $\ul d$ ranges 
over non-negative $n$-tuples $(d_1,\ldots,d_n) \in \bN^{\times n}$.  
We refer to $k[x_1,\ldots,x_n]$ as the {\it coordinate algebra} of 
$\bA^n$.  For any $\ul a \in \bA^n$, the value of $p(\ul x)$ on $\ul a$ is 
$\sum_{\ul d} c_{\ul d} \ul a^{\ul d}$. 

The {\bf Hilbert Nullstellensatz} tells us that $p(\ul x)$ is the 0 
polynomial (i.e., equals  $0 \in k[x_1,\ldots,x_n]$) if and only if 
$p(\ul a) = 0$ for all $\ul a \in \bA^n$.  This has a general formulation
which applies to any quotient $A = k[x_1,\ldots,x_n]/I$ with no nonzero
 nilpotent elements: $\ol p(\ul x)
\in A$ is 0 if and only if $p(\ul a) = 0$ for all $\ul a \in \bA^n$ which
satisfy $q(\ul a) = 0$ for all $q \in I$.

A closed subvariety of $\bA^n$ is the zero locus of a set $S$ of polynomials,
$Z(S) \subset \bA^n$  Let $<S>$ denote the ideal generated by $S$ and let
$I_S$ denote the radical ideal of all $g \in k[x_1,\ldots,x_n]$ for 
which some power of $g$ lies in $<S>$.  Then $Z(S) = Z(I_S)$ and 
$k[x_1,\ldots,x_n]/I_S$
is the ring of equivalence classes of polynomials $p(\ul x)$ for the 
equivalence relation $p(\ul x) \ \sim \  q(\ul x)$ if and only if 
$p(\ul x) - q(\ul x)$ vanishes on every $\ul a \in Z(S)$.   We say that
$A = k[x_1,\ldots,x_n]/I_S$ is the coordinate algebra of the affine algebraic
variety $\Spec A$ with underlying space $Z(S)$; the closed subsets
of $Z(S)$ are the subsets $Z(T)$ with $T \supset S$.  Thus, there
is a natural bijection between the closed subsets of $\bA^n$ (i.e., the
zero loci $Z(S) = Z(I_S)$) and their coordinate rings $k[x_1,\ldots,x_n]/I_S$
of algebraic functions.

More generally, the data of an affine $k$-scheme (of finite type over $k$) is a 
commutative, finitely generated $k$ algebra given non-uniquely as the quotient for
some $n > 0$ of $k[x_1,\ldots,x_n]$ by some ideal $J$.  An affine scheme
determines a functor from the category of commutative, finitely generated 
$k$-algebras to sets.  For $A = k[x_1,\ldots,x_n]/J$, this functor sends $R$ to 
$Hom_{k-alg}(A,R)$; in other words, $Hom_{k-alg}(A,R)$ equals the set of all 
$n$-tuples $\ul r \in R^n$ with the property that $p(\ul r) = 0$ for all 
$p(\ul x) \in J$.

Since the Hilbert Nullstellensatz does not apply to an affine $k$-scheme $A$
containing nilpotent elements,  we use the Yoneda Lemma to conclude the identification
of an affine scheme with its associated functor; thus, we may abstractly
define an affine scheme as a representable functor from finitely generated
commutative $k$-algebras to sets.


\subsection{Affine group schemes}

As made explicit in Definition \ref{defn:linearalg}, a linear algebraic group over
$k$ is, in particular, an algebraic variety over $k$.  
We introduced affine schemes in Subsection I.1 whose coordinate algebra might
have nilpotent elements in order to consider Frobenius kernels $G_{(r)}$ of 
linear algebraic groups (see Definition \ref{defn:Frobker}).  

\begin{defn}
\label{defn:linearalg}
A linear algebraic group $G$ over $k$ is an affine scheme over $k$ whose 
associated functor is a functor from finitely generated commutative
$k$-algebras to groups such that the coordinate algebra $k[G]$ of $G$ is an integral 
domain..
\end{defn}

For example, the linear algebraic group $GL_N$ is the affine variety of
$N \times N$ invertible matrices, with associated coordinate algebra
is $k[x_{1,1},\ldots,x_{n,n},z]/(det(x_{i,j}\cdot z-1)$.  As a functor,
$GL_N$ sends a commutative $k$-algebra $R$ to the group of $N\times N$
matrices with entries in $R$ whose determinant is invertible in $R$ (with
group structure given by multiplication of matrices).

We denote $GL_1$ by $\bG_m$, the multiplicative group.  The coordinate algebra
$k[\bG_m]$ of $\bG_m$ is the polynomial algebra 
$k[x,x^{-1}] \simeq k[x,y]/(xy-1)$.  The associated functor sends $R$ to 
the set of invertible elements $R^\times $ of $R$ with group
structure given by multiplication.

An even simpler, and for that reason more confusing, example is  
$\bG_a$, the additive group.  The coordinate algebra $k[\bG_a]$ of $\bG_a$
is $k[T]$.  The associated functor sends $R$ to itself, viewed as an 
abelian group (forgetting the multiplicative structure).

\begin{defn}
An affine group scheme over $k$ is an affine $k$-scheme whose 
associated functor is a functor from finitely generated commutative
$k$-algebras to groups.
\end{defn}

We shall use an alternate formulation of affine group schemes, in addition
to the formulation as a representable functor with values in groups.  
This formulation can be phrased geometrically as follows:  an affine 
group scheme is a group object in the category of schemes.

To be more precise, we state this formally.

\begin{defn}
An affine group scheme $G$ (over $k$) is the spectrum associated to a finitely generated,
commutative $k$-algebra $k[G]$ (the coordinate algebra of $G$) equipped
with a coproduct $\Delta_G: k[G] \to k[G] \otimes k[G]$ such that 
$(k[G],\Delta_G)$ is a Hopf algebra.
\end{defn}

This coproduct gives the functorial group structure on the $R$-points $G(R) \equiv Hom_{k-alg}(k[G],R)$
of $G$ for any finitely generated commutative $k$-algebra $R$:  namely, 
composition with $\Delta_G$ determines 
$$Hom_{k-alg}(k[G],R) \times Hom_{k-alg}(k[G],R) \simeq 
Hom_{k-alg}(k[G] \otimes k[G],R)$$
$$ \hskip 2in \to \ Hom_{k-alg}(k[G],R).$$
For example, the coproduct on the coordinate algebra of $GL_N$ is defined 
on the matrix function $X_{i,j} \in k[GL_N]$ by
$\Delta_{GL_N}(X_{i,j}) = \sum_\ell X_{i,\ell}\otimes X_{\ell,j}$.

\subsection{Characteristic $p > 0$}

In this subsection, we convey some of the 
idiosyncrasies of characteristic $p$ algebraic geometry.  We have already
mentioned one:  the Frobenius kernels $G_{(r)}$ of a linear algebraic
group $G$ are defined only if the ground field $k$ has positive characteristic.
Unlike the remainder of the text, in this subsection we allow $k$ to denote an arbitrary field
of characteristic $p$ (e.g., a finite field).

Let's begin by mentioning a few examples of fields of characteristic $p$,
where $p$ is a fixed prime number.  For any power $q = p^d$ of $p$, there
is a field (unique up to isomorphism) with exactly $q$ elements, denoted 
$\bF_q$.  For any set of ``variables" $S$ and any $k$, there is the
field (again, unique up to isomorphism) of all quotients $p(\ul s)/q(\ul s)$
of polynomials in the variables in $S$ and coefficients in $k$ such that $q(\ul s)$ is not
the 0 polynomial.  Typically, we consider a finite set $\{ x_1,\ldots,x_n \}$
of variables; in this case, we denote the field $k(x_1,\ldots,x_n)$.
If $I \subset k[x_1,\ldots,x_n]$ is a prime ideal,
then $k[x_1,\ldots,x_n]/I$ is an integral domain with field of fractions
$frac(k[x_1,\ldots,x_n]/I)$ of transcendence degree over $k$ equal to the
dimension of the affine algebraic variety associated to $k[x_1,\ldots,x_n]/I$. 

The key property of a field $k$ of characteristic $p$, and more generally of 
a commutative $k$-algebra $A$, is that $(a + b)^p = a^p + b^p$ for all $a,b \in A$.
The $p$-th power map $(-)^p: A \to A, \ a \mapsto a^p$ is thus a ring 
homomorphism.  However, if $a \in k$ does not lie in $\bF_p$ and 
if $b$ is such that $b^p \not= 0$,
then $(-)^p(a\cdot b) \ \not= \  a\cdot (-)^p(b)$ as would be required by 
$k$-linearity.

The (geometric) Frobenius map $F: k[x_1,\ldots,x_n] \to k[x_1,\ldots,x_n]$ is
a map of $k$-algebras (i.e., it is a $k$-linear ring homomorphism) defined
by sending an element $a \in k$ to itself, sending any $x_i$ to $x_i^p$.
Thus $F(\sum_{\ul d} c_{\ul d} \cdot \ul x^{\ul d}) = 
\sum_{\ul d} c_{\ul d} \cdot x^{p\cdot \ul d}$.  Viewed as a self-map of affine
space $\bA^n$, $F: \bA^n \to \bA^n$ sends the $n$-tuple $(a_1,\ldots,a_n)$
to the $n$-tuple $(a_1^p,\ldots,a_n^p)$ (in other words, the inverse image
of the maximal ideal $(x_1-a_1,\ldots,x_n-a_n)$ is the maximal ideal
$(x_1-a_1^p,\ldots,x_n-a_n^p)$).  

\begin{defn}
\label{defn:Atwist}
Let $A$ be a finitely generated commutative $k$-algebra and express
$A$ in terms of generators and relations by $k[x_1,\ldots,x_n]/(p_1,\ldots p_m)$.
For any $p(\ul x) = \sum_{\ul d}  c_{\ul d} \cdot \ul x^{\ul d} \in k[x_1,\ldots,x_n]$, set 
$\phi(p(\ul x)) \ = \ \sum_{\ul d} c_{\ul d}^p \cdot \ul x^{\ul d} \in k[x_1,\ldots,x_n]$; thus \\
$\phi: k[x_1,\ldots,x_n] \ \to \ k[x_1,\ldots,x_n]$
is an isomorphism of algebras which is semi-linear over $k$.

We define \ $A^{(1)} \ = \ k[x_1,\ldots,x_n]/(\phi(p_1),\ldots \phi(p_m)).$

We define the Frobenius map to be the $k$-linear map given by
$$F: A^{(1)} \ \to \ A, \quad \ol x_i \mapsto (\ol x_i)^p, $$
where $\ol x_i$ is the image of $x_i$ under either the projection $k[x_1,\ldots,x_n] \twoheadrightarrow A^{(1)}$
or the projection $k[x_1,\ldots,x_n] \twoheadrightarrow A$.
Hence, if the ideal $(p_1,\ldots p_m)$ is generated by elements in $\bF_p[x_1,\ldots,x_n]$ (i.e., 
if $A$ is defined over $\bF_p$), then the Frobenius map is an endomorphism $F: A \to A$.
\end{defn}

To verify that $F: A^{(1)} \ \to \ A$ is well defined, we observe that $F(\phi(p(\ul x)) = (p(\ul x))^p$
for any $p(\ul x) \in k[x_1,\ldots,x_n]$, so that $F((\phi(p_1),\ldots, \phi(p_m))) \subset (p_1,\ldots p_m)$.

An intrinsic way to define $A^{(1)}$ is given in \cite{FS}, which gives a quick way to show that 
the definition of $A^{(1)}$ does not depend upon generators and relations.  Namely, $A^{(1)}$
is isomorphic as a $k$-algebra to the base change of $A$ via the map $\phi: k \to k$ sending 
$a \in k$ to $a^p$.

One of the author's favorite constructions is the following construction of Serge Lang (\cite{L})
using the Frobenius.  Namely, if $G$ is an affine group scheme over $k$ which is defined over $\bF_p$,
then we have a morphism of affine $k$-schemes (but not of affine group schemes)
$$1/F: G  \stackrel{id \times F}{\to} G\times G 
\stackrel{id \times inv}{\to} G\times G \stackrel{\mu}{\to} G.$$
If $G$ is a simple algebraic group over $k$, then $G$ is defined over $\bF_p$ and $1/F$ is a covering
space map of $G$ over itself (i.e., $1/F$ is finite, etale), a 
phenomenon which is not possible for Lie groups or linear algebraic groups over
a field of characteristic $0$.

We conclude this subsection with the example most relevant for
our purposes, namely the example of Frobenius kernels.

\begin{defn}
\label{defn:Frobker}
Let $G$ be a linear algebraic group over $k$.  Then for any positive integer
$r$, we define the $r$-th Frobenius kernel of $G$ to be the affine group
scheme defined as the kernel of the $r$-th iterate of Frobenius, $F^r: G \to G^{(r)}$.
The functor associated to $G_{(r)}$ sends a finitely generated commutative
$k$-algebra $R$ to the kernel of the $r$-th iterate of the Frobenius,
$F^r: G(R) \to G^{(r)}(R)$.
\end{defn}

So defined, the coordinate algebra $k[G_{(r)}]$ of $G_{(r)}$ is the
quotient of $k[G]$ by the $p^r$-th power
of the maximal ideal of the identity of $G$.  (This quotient is well defined
for any field, but is a Hopf algebra if and only if $k$ is of characteristic $p$.)

A good example is the $r$-th Frobenius kernel of $GL_N$.  We identify
the functor $R \mapsto  GL_{N(r)}(R)$ 
as sending $R$ to the group (under multiplication) of $N\times N$ matrices with
coefficients in $R$ whose $p^r$-th power is the identity matrix.  In
characteristic $p$, if $A, B$ are two such $N\times N$ matrices, then
$(A\cdot B)^{p^r} = (A^{p^r})\cdot (B^{p^r})$, so that $GL_{N(r)}(R)$ is
indeed a group.

\subsection{Restricted Lie algebras}

In his revolutionary work on continuous actions (of Lie groups on real vector
spaces), Sophus Lie showed that the continuous action of a Lie group is
faithfully reflected by its ``linearization", the associated action of
its Lie algebra.   We may view the Lie algebra of a Lie group as the 
tangent space at the identity equipped with a Lie bracket on pairs of tangent
vectors which is a ``first
order infinitesimal approximation" of the commutator of pairs of elements of the
group.  Exponentiation sends a Lie algebra map to a neighborhood of
the identity in the Lie group.

This property of the Lie algebra to faithfully reflect the action of the Lie group
fails completely in our context of representation theory of affine
group schemes over a field of characteristic $p$.  Instead, one should
consider all ``infinitesimal neighborhoods" $G_{(r)}$ of the identity of $G$.
Nevertheless, the Lie algebra $\fg$ of $G$ and its representations play 
a central role in our considerations.

\begin{defn}
\label{def:liealgebra}
A Lie algebra $\fg$ over $k$ is a vector space equipped with a binary operation 
$[-,-]: \fg \otimes \fg \to \fg$ satisfying $[x,x] = 0$ for
all $x\in \fg$ and the Jacobi identity
$$[x,[y,z]] + [z,[x,y]] + [y,[z,x]] \ = \ 0, \quad \forall x,y,z \in \fg.$$

A Lie algebra is $p$-restricted if it has an additional ``$p$-operation" 
$[-]^{[p]}: \fg \to \fg$ which satisfies conditions (see \cite{Jac}) satisfied
by the $p$-th power of matrices in $\gl_N = Lie(GL_N)$ and by the $p$-th power 
of derivations of algebras (over a field of characteristic $p$).
\end{defn}

Any finite dimensional restricted Lie algebra $\fg$ admits an embedding as
a Lie algebra into some $\gl_N$ such that the $p$-operation of $\fg$ is the
restriction of the $p$-th power in $\gl_N$.  The subtlety here is that a 
Lie algebra is not equipped with an associative multiplication (except, 
accidentally, for $\gl_N$).   If $\fg \subset \gl_N$ is an embedding of $p$-restricted
Lie algebras, then the $p$-th power in $\gl_N$ of an element $X \in \fg$ is again
in $\fg$ and equals $X^{[p]} \in \fg$.

Given an affine group scheme $G$ over $k$, the Lie algebra $\fg$ of $G$ can
be defined as the space of $G$-invariant derivations of $k[G]$, a Lie subalgebra
of the Lie algebra of all $k$-derivations of $k[G]$.  Alternatively, $\fg$ can be
identified with the vector space of $k$ derivations $X: k[G] \to k[G]$ based at the 
identity $e \in G$; in other words, elements of $\fg$ can be viewed as 
$k$-linear functionals on $k[G]$ satisfying
$X(f\cdot h) = f(e)X(h) + h(e)X(f)$, with bracket $[X,Y]$  defined
to be the commutator $X\circ Y - Y \circ X$.  Because we are working over
a field of characteristic $p$, the $p$-fold composition of such a derivation
$X$ with itself is again a derivation based at $e$; sending $X$ to this
$p$-fold composition, $X \mapsto X\circ \cdots \circ X$, equips $\fg$
with a $p$-operation.

In other words, $Lie(G)$ is a $p$-restricted Lie algebra.  For example,
for $G = \bG_a$, the associated $p$-restricted Lie algebra $\fg_a$ is
the 1-dimensional vector space $k$ (whose bracket necessarily is 0)
and the $p$-operation sends any 
$c \in k$ to 0.  For $G = \bG_m$, the associated Lie algebra is again
the 1-dimensional vector space with trivial bracket, but the $p$-operation sends $a \in k$ to
$a^p$.

As a lead-in to Lecture II, we recall the definition of a $p$-restricted 
representation of a restricted Lie algebra $\fg$.  The ``differential" of
a representation of a group scheme over $k$ is a $p$-restricted representation
of $\fg = Lie(G)$

\begin{defn}
\label{defn:restricted}
Let $\fg$ be a restricted Lie algebra over $k$.  A $p$-restricted representation
of $\fg$ is a $k$-vector space $V$ together with a $k$-bilinear pairing
$$\fg \otimes V \ \to \ V, \quad (X,v) \mapsto X(v)$$
such that $[X,Y](v) \ = \ X(Y(v)) - Y(X(v))$ and $X^{[p]}(v)$ equals the 
result of iterating the action of $X$ $p$-times, $X(X(\cdots X(v))\cdots)$.  

Let $U(\fg)$ denote the universal enveloping algebra of $\fg$, defined
as the quotient of the tensor algebra $T^*(\fg) = \oplus_{n \geq 0} \fg^{\otimes n}$
by the ideal generated by the relations $X\otimes Y - Y\otimes X - [X,Y]$
for all pairs $X,Y \in \fg$.  Then the restricted enveloping algebra of
$\fg$, denoted here as in \cite{J} by $U^{[p]}(\fg)$, is the quotient of 
$U(\fg)$ by the ideal generated by the relations $X^{\otimes p} - X^{[p]}$
for all $x \in \fg$.  If $\fg$ has dimension $n$ over $k$, then $U^{[p]}(\fg)$ 
is a finite dimensional $k$-algebra of dimension $n^p$.  

A structure of a $p$-restricted representation of $\fg$ on a $k$-vector space
$V$ is naturally equivalent to a $U^{[p]}(\fg)$-module 
structure on $V$.
\end{defn}

A good example is the ``adjoint representation" of a restricted Lie 
algebra $\fg$.  Namely, we define $\fg \otimes \fg \to \fg$ sending 
$(X,Y)$ to $X(Y) \equiv [X,Y]$.  The Jacobi identity of $\fg$ implies the 
condition that $[X_1,X_2](Y) = X_1(X_2(Y)) - X_2(X_1(Y))$ and the axioms of 
a $p$-operation imply that $X^{[p]}(Y) = [X,[X,\ldots [X,Y]\ldots]$.

\vskip 1in


\section{Lecture II: Algebraic representations}

Following Lecture I which discussed finite groups, restricted Lie algebras, 
Frobenius kernels, and algebraic groups (all of which we would include
under the rubric of ``group-like structures"), this lecture discusses what
are the {\bf algebraic} representations of these objects.   Our basic reference
for this lecture is the excellent book ``Representations of Algebraic Groups"
by J. Jantzen \cite{J}.

	Here is the outline provided to participants attending this second
lecture.

\vskip .1in
\noindent
II.A    Equivalent formulations of rational $G$-modules.
       
        i.) For $M$ finite dimensional, matrix coefficients;
        
        ii.)  Functorial actions;
        
        iii.) Comodules for coalgebra;
        
        iv.) Locally finite modules for hyperalgebra.

\vskip .1in
\noindent
II.B    Examples.
       
        i.)  $\bG_a$-modules, $\bG_m$-modules;
       
        ii.)  Modules arising from (strict polynomial) functors;
        
        iii.) Induced modules;
        
        iv.)  Abelian category.

\vskip .1in
\noindent
II.C   Weights arising from action of a torus;
       
        i.)  Borel's theorem about stable vector for $B$ solvable;
       
        ii.)  Highest weight of an irreducible;
       
        iii.)  $H^0(\lambda)$ and Weyl character formula;
       
        iv.)  Lusztig's Conjecture.

\vskip .1in
\noindent
II.D    Representations of Frobenius kernels.
        
        i.)  General theory of representations of Artin algebras
             (e.g., Wederburn theorem; injective = projective);
        
        ii.)  Special case of $G_{a(r)}$;
        
        iii.)  $G$-modules and $\{ G_{(r)} \}$-modules.

\vskip .1in
\noindent
Topics for discussion/project:
\vskip .1in
\noindent
II.A  Working out diagrams for checking properties of coaction.
      Examples of $GL_N$-actions which are not algebraic
      Investigating the action of the Lie algebra on $M$ associated to
         a rational action of $G$ on $M$.

\vskip .1in
\noindent
II.B  Working out properties for the categories of finite dimensional and
         all rational $G$-modules.
      Frobenius reciprocity.

\vskip .1in
\noindent
II.C  Discussion of roots for a simple algebraic group.  
      Understanding of Weyl's character formula (for complex repns).

\vskip .1in
\noindent
II.D  Expanded investigation of Artin algebras.
      Discussion of representations of $kE$, $E$ elementary abelian.
      Lie algebra actions.
      
\vskip .2in

\subsection{Algebraic actions}
      
      We are interested in the {\bf algebraic} actions of the ``group-like"  
 structures $G$ discussed
 in the previous lecture on vector spaces $V$ over our chosen field $k$ 
 (which we take to be algebraically closed of characteristic $p$ for some
 prime $p$).   A group action of $G$ on $V$ is a pairing
 $$\mu: G \times V \to V, \quad (g,v) \mapsto \mu(g,v)$$
 whose ``adjoint" is the corresponding group homomorphism $\rho_\mu: G \to Aut_k(V)$.
 For simplicity, we first assume that $V$ is of some finite dimension $N$, 
 so that $\rho_\mu$ takes the form
 $$\rho_\mu: G \  \to \ Aut_k(V) \simeq  \ GL_N.$$
 A discrete action is one for which no further requirement on 
 $\rho_\mu$ is imposed other than it be group homomorphism
 (on the $k$-points of $G$ and $GL_N$, thus of the form $G(k) \to GL_N(k)$).  
 A continuous action has the additional condition that
 composition with each matrix function 
 $$X_{i,j} \circ \rho_\mu: G(k) \ \simeq  \ GL_N(k) \ \to k, \quad 1 \leq i,j \leq N$$
is continuous;  for this to be meaningful, $k$ must have a topology (e..g., for the fields
 $\bR, \ \bC$ which are of course of characteristic 0).
 
 Recall that the data of an affine scheme $X$ (e.g., an affine group scheme)  is equivalent
 to that of its coordinate algebra $k[X]$, often called its algebra of ``regular functions".
 We view elements of $k[X]$ as ``algebraic functions" from $X$ to $k$; more formally,
 an algebraic function is a function functorial with respect to maps of finitely
 generated $k$-algebras.  In other words,  $f \in k[X]$ is equivalent to the following data: for
 any finitely generated commutative $k$-algebra $A$, a map of sets 
$Hom_{k-alg}(k[X],A) \to A$ (i.e., a function from the $A$ points of $X$ to $A$)
which is functorial with respect to $A$.    (Observe that
$f \in k[X]$ is recovered from this data as the image of the identity
$Hom_{k-alg}(k[X],k[X]) \to k[X]$;  for any $f$ and any $A$, we send
$\phi \in Hom_{k-alg}(k[X],A)$ to the $\phi(f) \in A$.)
 
 Before we formulate the definition of an algebraic action of a general affine
 algebraic group $G$ on a $k$-vector space, we first consider algebraic actions of 
 a linear algebraic group.  The definition below implicitly uses  the Hilbert Nullstellensatz.

 \begin{defn}
 \label{defn:algebraic}
 Let $G$ be a linear algebraic group over $k$ and $V$ a finite dimensional
 $k$-vector space of dimension $N$.  Then an action  $\mu: G\times V \to V$
 of $G$ on $V$ is defined to be algebraic (usually called ``rational") if each matrix coefficient of $\mu$, 
 $$X_{i,j} \circ \rho_\mu: G(k) \ \to  \ GL_N(k) \ \to k, \quad 1 \leq i,j \leq N,$$ 
is an element of $k[G]$. 
 \end{defn}

 \begin{ex}
 \label{ex:defining}
 We give a first example of an algebraic action.  Further examples will easily
 follow from alternative formulations in Proposition \ref{prop:alternate} of the 
 algebraicity condition of Definition \ref{defn:algebraic}.
 
 Let $G = GL_n$ and let $V$ be the elements of degree $d$ in the polynomial algebra $k[x_1,\ldots,x_n]$.
 We define the group action 
 $$\mu: GL_n(k) \times V \to V, \quad g \cdot  (x_1^{d_1}\cdots x_n^{d_n}) \ = 
 \ (g\cdot x_1)^{d_1} \cdots (g\cdot x_n)^{d_n},$$
 where $\sum_i d_i = d$ and where $g\cdot x_j = \sum_j X_{i,j}(g)x_i$.
 Thus, $V$ is the $d$-fold symmetric power $S^d(k^n)$ of the ``defining representation" of $GL_n$ on 
 $k^n$.
 
 It is a good (elementary) exercise to verify that each matrix coefficient of $\mu$ is an element of $k[GL_n]$.
 
 We can argue similarly for exterior powers $\Lambda^d(k^n)$. For example,  $\Lambda^n(k^n)$ is
 a 1-dimensional representation of $GL_n$ given by
 $$\mu: GL_n(k) \times k \ \to \ k, \quad g\cdot v \ = \ det(g) v;$$
 the algebraicity condition is simply that $\rho_\mu: GL_n(k) \to GL_1(k) = k^\times \subset k$
 is an element of $k[GL_n]$.
 Observe that this representation is ``invertible", in the sense that
 $\mu^{-1}: GL_n(k) \times k \ \to \ k, \quad g\cdot v \ = \ det(g)^{-1} v$
 is also algebraic.
  \end{ex}
 
We extend the definition of an algebraic action to encompass an affine group scheme over $k$ acting
on an arbitrary $k$-vector space. 
  
  \begin{prop}
  \label{prop:alternate}
  Let $G$ be an affine group scheme over $k$ and $V$ a $k$-vector space.  Then the following
  two conditions on a group action $\mu: G(k) \times V \to V$ are equivalent.
  \begin{enumerate}
  \item
  There exists a $k$-linear map $\Delta_V: V \to V\otimes k[G]$ which provides $V$ with the structure of a 
  $k[G]$-comodule; the pairing $\mu: G(k) \times V \to V$ is given by sending $(g,v)$ to 
  $((1 \otimes ev_g) \circ \Delta_V)(v)$.
  \item
  There exists  a pairing of functors $\ul \mu: G(-) \times  ((-)\otimes V) \to (-)\otimes V$ on commutative $k$-algebras
  such that  $G(A) \times (A \otimes V) \to A\otimes V$ is an $A$ linear action of $G(A)$ for any commutative
   $k$-algebra $A$; the pairing $\mu: G(k) \times V \to V$ is given by taking $A = k$.
  \end{enumerate}
  Moreover, if is $V$ is finite dimensional, say of dimension $N$, then the first two conditions are equivalent
  to the condition:

\quad (3) The adjoint of $\mu$ is a map $G \to GL_N$ of group schemes over $k$.

Furthermore, if $G$ is a linear algebraic group and $V$ has dimension $N$, then these equivalent 
conditions are equivalent to the algebraicity condition of Definition \ref{defn:algebraic}.
  
  We define an algebraic action of $G$ on an arbitrary vector space $V$ over $k$ to be one 
  that satisfies the equivalent conditions (1) and (2).
  \end{prop}

  \begin{remark}
  We say that a group scheme $G$ over $k$ is a finite group scheme over $k$ if $k[G]$ is 
  finite dimensional (over $k$).  For any finite group scheme $G$ over $k$ and any $k$-vector
  space $V$, there is a natural bijection between comodule structures $\Delta_V: V \to V\otimes k[G]$
  and module structures $(k[G])^\# \otimes V \to V$.  Namely, we associate to $\Delta_V$ the
  pairing 
  $$(k[G])^\# \otimes V \stackrel{1 \otimes \Delta_V}{\to} (k[G])^\# \otimes V \otimes k[G] \to V,$$
  where the second map is given by the evident evaluation $(k[G])^\# \otimes k[G] \to k$.
   \end{remark}
   
\begin{note}
If $G$ is a finite group scheme over $k$, we denote by $kG$ the algebra $(k[G])^\#$ and refer
to $kG$ as the group algebra of $G$.  In \cite{J}, $kG$ is called the distribution algebra of $G$
(of $k$-distributions at the identity) whenever $G$ is an infinitesimal group scheme (i.e., 
whenever $G$ is a connected, finite group scheme).

If $G$ is a linear algebraic group over $k$ we denote by $kG$ the colimit $\varinjlim_r kG_{(r)}$
and refer to this algebra as the group algebra of $G$; once again, this is called the distribution
algebra of $G$ by Jantzen in \cite{J}; it also is called the hyperalgebra of $G$ by many authors
(e.g., \cite{Sull}).
\end{note}
 \vskip .2in
  
\subsection{Examples}

  Now, for some more examples.
  
\begin{ex}
\label{ex:groupalg}
\begin{enumerate}
\item  Take $G$ to be any affine group scheme.  Then the coproduct $\Delta_G: k[G] \ \to \ k[G] \otimes k[G]$ 
determines the right regular action $\mu: G \times G \to G$ (where the first factor of $G \times G$
is the object acted upon and the second factor is the group acting).

\item
Take $G = \bG_{a(r)}$ for some $r > 0$.  Then $k[\bG_{a(r)}]$ equals $k[T]/T^{p^r}$ with linear dual
 $kG_{a(r)} \ = \ k[u_0,\ldots,u_{r-1}]/(\{ u_i^p\})$; we identify $u_i$ as the $k$-linear map sending
 $T^n$ to 0 if $n \not= p^i$ and sending $T^{p^i}$ to 1.
Since $k[u_0,\ldots,u_{r-1}]/(\{ u_i^p\})$ can be identified with the group algebra of the elementary abelian $p$-group 
$(\bZ/p\bZ)^{\times r}$, 
we conclude a equivalence of categories between the category of $\bG_{a(r)}$-representations and the category of representations of 
$(\bZ/p\bZ)^{\times r}$ on $k$-vector spaces.

\item
Take $G = \bG_a$, with $k[\bG_a] = k[T]$ and consider
$$k\bG_a \ \equiv \ \varinjlim_r (k[G_{a(r)}])^\#) \ = \ k[u_0,\ldots,u_n, \ldots]/(\{ u_i^p, i \geq 0\}).$$
Then an algebraic action of $\bG_a$ on $V$ is equivalent to the data of infinitely many $p$-nilpotent operators
$u_i: V \to V$ which pair-wise commute such that for any $v \in V$ there 
exist only finitely many $u_i$'s with $u_i(v) \not= 0$.

 \item
Take $G = \bG_m$, with coordinate algebra $k[\mathbb G_m] \simeq k[T,T^{-1}]$. 
A $k[\bG_m]$-comodule structure on $V$ has the form
$$\Delta_V: V \to V\otimes k[\bG_m], \quad v \mapsto \sum_{n \in \bZ} p_n(v) \otimes T^n.$$
where each $p_n: V \to V$ is a $k$-endomorphism of $V$.   One checks that $\sum_n p_n = id_V$, $p_m \circ p_n = \delta_{m,n} p_n$
which implies that $V \ = \ \bigoplus_{n\in \bZ} V_n$ where $V_n = \{ v \in V: \Delta_V(v)  = v \otimes T^n \}$.
For $v \in  V_n$, $a \in \bG_m(k) = k^\times$ acts by sending $v$ to $T^n(a)\cdot v = a^n \cdot v$.  In particular,
$V_n$ is a direct sum of 1-dimensional irreducible $\bG_m$-modules whose isomorphism class is characterized
by $n \in \bZ$, the power through which $k^\times$ acts.  

It is useful to view the action of $\bG_m$ on some 1-dimensional irreducible $\bG_m$-module as the composition of a homomorphism $\lambda: \bG_m \to \bG_m$ with the defining action of $\bG_m$ on $k$.  Such a homomorphism
(or character) is given by a choice of $n \in \bZ$ (corresponding to the map on coordinate algebras $k[T,T^{-1}] \to
k[T,T^{-1}]$ sending $T$ to $T^n$).  See Definition \ref{defn:character} below.

\item
Take $G = GL_n$ and fix some $d > 0$.  Consider $\rho: GL_n \to GL_N$ (corresponding to 
an action of $GL_n$ on a vector space of dimension $N$) with the property that 
$X_{i,j} \circ \rho: GL_n \to k$
extends to a function $GL_n \subset \bA^{n^2} \to k$ which is a homogeneous polynomial of degree
$d$ in the $n^2$ variables of $\bA^{n^2}$ for some $d > 0$ independent of $(i,j)$.   
Such an action is said  to be a polynomial representation 
homogeneous of degree $d$ of $GL_n$ (of rank $N$).  This generalizes the examples of 
Example \ref{ex:defining}.
\end{enumerate}
\end{ex}
  
We next recall the definition of the character group $X(G)$ of $G$, extending the discussion of
Example \ref{ex:groupalg}.4.  For our purposes, the diagonalizable affine group schemes of most 
interest are (split) tori $T$ (isomorphic to some product of $\bG_m$'s) and their Frobenius kernels.
 
 \begin{defn}
 \label{defn:character}
Let $G$ be an affine group scheme over $k$.  A character of $G$ is a homomorphism
of group schemes over $k$, $\lambda: G \to \bG_m$.  Using the abelian group structure of $\bG_m$, 
the set of characters of $G$ inherits an abelian group structure which is denoted by $X(G)$.

An affine group scheme $G$ is said to be diagonalizable if its coordinate algebra $k[G]$ 
is isomorphic as a Hopf algebra to the group algebra $k\Lambda$, where $\Lambda = X(G)$
is the character group of $G$.  (Here, the coproduct on $k\Lambda$ is given by $\lambda \mapsto
\lambda \otimes \lambda$.)
 \end{defn}
 
 For example, $\bG_m$ is a diagonalizable group scheme over $k$ with coordinate
 algebra $k[\bG_m] \simeq k\bZ$.
 
 \begin{prop}
 \label{diagonalizable}
 Let $G$ be a diagonalizable group scheme with character group $\Lambda$.  Then
 an algebraic representation of $G$ on a $k$-vector space $V$ has a natural 
 decomposition as a direct sum, $V \ \simeq \ \bigoplus_{\lambda \in \Lambda} V_\lambda$,
 where $V_\lambda = \{ v \in V: g\cdot v = \lambda(g)\cdot v, \forall g \in G \}$.
 \end{prop}
 
 One important construction which produces algebraic representations is ``induction
 to $G$ from a closed subgroup $H \subset G$".  This is sometimes called 
 ``co-induction" by ring theorists.
 
 \begin{defn}
 \label{defn:induced}
 Let $G$ be an affine group scheme and $H \subset G$ a closed subgroup scheme
 (i.e., the coordinate algebra of $H$ is the quotient of $k[G]$ by a Hopf ideal).
 Let $H \times W \to W$ be an algebraic representation of $H$.  Then the 
 induced representation $ind_H^G(W)$ has underlying vector space given by
 $(k[G]\otimes W)^H$, the elements of $k[G]\otimes W$ fixed under the 
 diagonal action of $H$ acting on $k[G]$ through the right regular representation
 and on $W$ as given; the $G$ action $G \times (k[G]\otimes W)^H$ is given
 by the left regular representation on $G$.
 \end{defn}
 
 \vskip .2in
 
 \subsection{Weights for  $G$-modules}
 
 If $G$ is a linear algebraic group over $k$, then a Borel subgroup of $G$ is a maximal solvable,
 closed, connected algebraic subgroup.  With our standing hypothesis that $k$ is algebraically
 closed, all such Borel subgroups $B \subset G$ are conjugate in $G$.  Any maximal torus $T$ of $G$
 (i.e., a product of $\bG_m$'s of maximal rank) is contained in some Borel $B \subset G$ and
 maps isomorphically onto the quotient of $B$ by its unipotent radical $U$; thus $B \  \simeq \ U \rtimes T$. 
 
 \begin{defn}
 Let $G$ be a linear algebraic group, $T \subset G$ a maximal torus, $\ell$ the rank of $T$ (so
 that $T \simeq \bG_m^{\times \ell})$.  Let $V$ be a $G$-module (i.e., 
 an algebraic representation of $G$ on the $k$-vector space $V$).  Then the set of 
 weights of $V$ are those characters $\lambda \in X(T)$ with the property that the
 decomposition of $V$ as a $T$-module has non-zero $\lambda$-eigenspace (i.e., $V_\lambda \not= 0$).
 \end{defn}
 
 If $G$ is unipotent (for example, the algebraic subgroup of $GL_N$ of upper triangular
 matrices with 1's on the diagonal), then its maximal torus is simply the identity group.
 However, for $G$ simple (or, more generally for $G$ reductive), this concept of the
 weights of a representation is the key to parametrizing the irreducible representations of $G$
 as stated in Proposition \ref{prop:irred}.
 
\begin{defn}
\label{defn:irred-ind}
Let $G$ be an affine group scheme over $k$.  A non-zero
$G$-module $V$ (given by an algebraic action $\mu: G \times V \to V$)
is said to be {\bf irreducible} if $V$ contains no non-trivial $G$ submodule;
in other words, the only $k[G]$-comodules contained in
$V$ are 0 and $V$ itself.

A  non-zero $G$-module $V$ is said to be indecomposable
 if there do not exist two non-zero $G$ submodules $V^\prime, V^{\prime\prime}$ of
 $V$ such that $V \simeq V^\prime \oplus V^{\prime\prime}$.
 \end{defn}
 
 We remind the reader that a reductive algebraic group over $k$ is a linear algebraic
 group whose maximal connected, normal, unipotent subgroup is trivial.  Every reductive
 algebraic group over $k$ is defined over $\bF_p$.
 
 \begin{prop}
 \label{prop:irred}
 Let $G$ be a reductive algebraic group over $k$, $B \subset G$ a Borel subgroup, and $T \subset B$ 
 a maximal torus.  
 There is a 1-1 correspondence between the dominant weights $X(T)_+ \subset X(T)$
 and (isomorphism classes of) irreducible $G$-modules.  Namely, to a dominant
 weight $\lambda$, one associates the irreducible $G$-module 
  \begin{equation}
 \label{eqn:Llambda}
 L_\lambda \ \equiv \ soc_G(ind_B^G(k_\lambda))
 \end{equation}
(where the socle of a $G$-module is the direct sum of all irreducible $G$-submodules).
 Here, $k_\lambda$ is the 1-dimensional $B$-module with algebraic action 
 $B \times k_\lambda \to k_\lambda$ sending $(b,a)$ to $\lambda(\ol b)a$, where $\ol b \in T$ 
 is the image of $b$ in the quotient $B \twoheadrightarrow T$ and $T \times k_\lambda \to k_\lambda$
 has adjoint $\lambda: T \to \bG_m$.  Moreover, the canonical
 map $ind_B^G(k_\lambda) \to k_\lambda$ identifies $k_\lambda$ with the (1-dimensional)
 $\lambda$-weight space of $ind_B^G(k_\lambda)$, and $\lambda$ is the unique highest
 weight of $ind_B^G(k_\lambda)$ and of $L_\lambda$.
 \end{prop}
 
 Although $ind_B^G(k_\lambda)$ as in Proposition \ref{prop:irred} is indecomposable, 
the inclusion $L_\lambda \ \subset ind_B^G(k\lambda)$ is an equality only for ``small" $\lambda$.
 
\vskip .2in

\subsection{Representations of Frobenius kernels}
 
 In this subsection, $G$ will denote a linear algebraic group over $k$. We briefly investigate the algebraic representations
 of the group scheme $G_{(r)} \ \equiv \ ker \{ F^r: G \to G^{(r)} \}$.
 
 \begin{prop}
 For any $r > 0$, the coordinate algebra $k[G_{(r)}]$ of $G_{(r)}$ is a finite dimensional, local (commutative)
 $k$-algebra.  Moreover, as a $G_{(r)}$ representation, $k[G_{(r)}]$ is isomorphic to its $k$-linear
 dual $kG_{(r)}$.  Consequently, in the category of $k[G_{(r)}]$-comodules (naturally isomorphic to the
 category of  $kG_{(r)}$-modules), an object is injective if and only if it is projective.
 \end{prop}
 
 For $r= 1$, $kG_{(1)}$ is isomorphic as an algebra to the restricted enveloping algebra
 of $Lie(G)$ (see Definition \ref{defn:restricted}).
 
  We remark that since $kG_{(r)}$ is a finite dimensional $k$-algebra which is injective as an algebra over itself,
 many standard techniques for studying the representation theory of Artin algebras over $k$ apply.
 Of course, $kG_{(r)}$ has more structure: it is a cocommutative Hopf algebra.
 
 For simplicity, assume $G$ is defined over the prime field $\bF_p$ (which means that the Hopf
 algebra $k[G]$ arises as the base change of a Hopf algebra over $\bF_p$, $k[G] = k\otimes_{\bF_p} \bF_p[G]$).
 This assumption enables us view the Frobenius map as an endomorphism of $G$, $F: G \to G$.

 \begin{defn}
 \label{Frobenius twist}
 Given an algebraic action \ $\mu: G \times V \ \to \ V$ \ of an affine group scheme $G$ on
 a $k$-vector space $V$, we define the (first) Frobenius twist $V^{(1)}$ of $V$ to be the $k$-vector space
 whose underlying abelian group equals that of $V$ and whose $k$-linear action is given by
  $c \cdot v^{(1)} \ = \ (c^p v^{(1)})^{(1)}$, for $c \in k, v^{(1)} \in V$; the algebraic action of $G$ on $V^{(1)}$ is defined as 
 $$\mu^{(1)} \circ (F\times 1): G\times V^{(1)} \ \to \ G^{(1)} \times V^{(1)} \ \to \ V^{(1)}.$$
  
 We inductively define $V^{(r+1)}$ to be $(V^{(r)})^{(1)}$ for any $r \geq 0$.
  
 Since $G_{(r)}$ is the kernel of $F^r$, we immediately conclude that the action of $G_{(r)}$ on $V^{(r)}$ 
 is trivial.
 \end{defn} 
  
 \begin{ex}
 \label{ex:Steinberg}
 Let $G$ be a simply connected, semi-simple algebraic group over $k$ and consider the irreducible $G$-module $L(\lambda)$ of 
 highest weight $\lambda$.  The Steinberg tensor product theorem \cite{St} asserts that 
 \begin{equation}
 L(\lambda) \ \simeq \ L(\lambda_1) \otimes L(\lambda_2)^{(1)} \otimes \cdots \otimes L(\lambda_s)^{(s)}
 \end{equation}
 where $\lambda = \sum_{i=0}^s p^i\lambda_i$ and each $\lambda_i$ is a $p$-restricted dominant weight.
 The condition that $\lambda$ be restricted is defined combinatorially, but is equivalent to the condition that 
 $L(\lambda)$ restricts to an irreducible $G_{(1)}$-module.
 
 In other words, each irreducible  $G$-module is a tensor product of Frobenius twists of $G$-modules which
 arise as irreducible restricted representations of $U^{[p]}(\fg)$.
  \end{ex}
  
  Example \ref{ex:Steinberg} emphasizes that restricting a $G$-module $V$ to  $kG_{(1)}$
  (we view this as taking the first order approximation of the $G$-action) loses enormous amount of
  information:  for example, irreducible $G$-modules $L_\lambda, \ L_{\lambda^\prime}$  have isomorphic restrictions
  to $kG_{(1)}$ if and only if $\lambda - \lambda^\prime$ can be written as a 
  difference of $p$-multiples of dominant weights.

 The following theorem of J. Sullivan in \cite{Sull} reveals the close connection of the representation theory of
 the family $\{ G_{(r)}, \ r > 0 \}$ of algebras with the rational representations of $G$.  Recall that if
 $A$ is a $k$-algebra and $M$ is an $A$-module, then $M$ is said to be locally finite if each finite
 dimensional subspace of $M$ is contained in some finite dimensional $A$-submodule of $M$.
 
 \begin{thm}
 Let $G$ be a simply connected, simple algebraic group over $k$.  Then there is an equivalence
 of categories between the category of  $G$-modules and locally finite modules for the 
 $k$-algebra $\varinjlim_r kG_{(r)}$.
 \end{thm}
 
 \vskip 1in
 

 \section{Lecture III: Cohomological support varieties}
 
 In this lecture we provide a quick overview of the theory of cohomological support varieties
 for finite groups, $p$-restricted Lie algebras, and finite group schemes.   In the lecture,
 the author discussed a comparison between one formulation of cohomological
 support varieties for linear algebraic groups and the theory discussed in the final lecture
 (i.e., Lecture IV) using 1-parameter subgroups.    In the text below, we briefly discuss 
 very recent computations for unipotent linear algebraic groups.
 
 We begin with the outline prepared in advance of the lectures, an outline which
 does not well summarize the text which follows.

\vskip .1in
\noindent
III.A   Indecomposable versus irreducible.
       
        i.)    examples of semi-simplicity;
        
        ii.)   examples of $(\bZ/p)^{\times n}$;
       
        iii.)  concept of wild representation type.
        
        \vskip .1in
\noindent
III.B   Derived functors.
        
        i.)    left exact functors, $(-)^G = Hom_{G-mod}(k,-)$;
       
        ii.)   injective resolutions and right derived functors;
        
        iii.)  $Ext_G^1(k,M)$;
        
        iv.)   representation of $Ext_G^i(k,M)$ as equivalence classes of
                  extensions.

\vskip .1in
\noindent
III.C   Commutative algebras and affine varieties.
        
        i.)    $\Spec A$, the prime ideal spectrum;
        
        ii.)   elementary examples;
       
        iii.)  $\Spec H^\bu(G,k)$;
        
        iv.)   (Krull) dimension and growth;
       
        v.)    $\Spec H^\bu(G,k)/ann(Ext_G^*(M,M))$;
       
        vi.)   Quillen's stratification theorem.
        
        vii.)  Carlson's conjecture for $G = (\bZ/p)^{\times n}$.

\vskip .1in
\noindent
III.D   Linear algebraic groups.
        
        i.)    $H^*(\bG_a,k)$;
        
        ii.)   $H^\bu(U_3,k)_{red}$;
        
        iii.)  Definitions of $V^{coh}(G), V^{coh}(G)_M$.

\vskip .1in
\noindent
Topics for discussion/projects:

\vskip .1in
\noindent
III.A   Presentation of finite/tame/wild representation type.
        Presentation of families of indecomposable $(\Z/p)^{\times 2}$-modules.

\vskip .1in
\noindent
III.B   Exposition of representation of $Ext_G^i(N,M)$ by extension classes.
        Discussion of other derived functors.
        Project on spectral sequences.

\vskip .1in
\noindent
III.C   Discussion of algebraic curves over $k$.
        Hilbert Nullstellensatz.
        Computation of $H^*((\Z/p)^{\times n},k)$.

\vskip .1in
\noindent
III.D    Open questions about detection modulo nilpotents and finite
             generation.
             
\vskip .1in

\subsection{Indecomposable versus irreducible}

We revisit the distinction between irreducible and indecomposable as defined
in Definition \ref{defn:irred-ind}.

Let $R$ be a (unital associative) ring and consider two left $R$-modules $M, N$.  Then an extension
of $M$ by $N$ is a short exact sequence $0 \to N \to E \to M \to 0$ of left $R$-modules.
We utilize the  equivalence relation on such extensions for fixed $R$-modules $M, N$
as the equivalence relation generated by commutative diagrams of $R$-modules 
of the form
\begin{equation}
\label{cocartesian}
\begin{xy}*!C\xybox{%
\xymatrix{
0 \ar[r] & N \ar[d]^= \ar[r] & E \ar[d] \ar[r] & M \ar[d]^= \ar[r] & \ 0 \\
0 \ar[r] & N  \ar[r] & E^\prime  \ar[r] & M \ar[r] & \ 0}
}\end{xy}
\end{equation}
relating the upper extension to the lower extension.
The set of such extensions of $M$ by $N$ form an abelian group
denoted $Ext_R^1(M,N)$.  
Cohomology groups (i.e., $Ext$-groups) at their most basic level 
are invariants devoted to detecting inequivalent extensions.
Rather than give information about basic building blocks (i.e.,
irreducible $R$-modules), cohomology can be used to show that a
pair of indecomposable $R$-modules with the same irreducible
``constituents" are not isomorphic.

For some purposes, one ``kills" such extensions by considering the 
Grothendieck group $K_0^\prime(R)$ defined as the free abelian 
group on the set of isomorphism classes of left $R$-modules 
modulo the equivalence relation $E \sim M \oplus N$ whenever
$E$ is an extension of $M$ by $N$.    This construction eliminates
the role of cohomology.    Said differently, if $R$ satisfies the condition
that every $R$-module splits as a direct sum of irreducible modules, then
(positive degree) cohomology groups $Ext_R^i(M,N)$ vanish.

Rather than consider an abelian category of $R$-modules, we shall
consider the abelian category $Mod_k(G)$ of  $G$-modules for an affine
group scheme $G$ over $k$.  If $k[G]$ is finite dimensional over $k$, then
$Mod_k(G)$ is isomorphic to the category $Mod(R)$ of left $R$-modules,
where $ R= kG$; for any affine group scheme, $Mod_k(G)$  
is equivalent to the abelian category of $k[G]$-comodules.

The representation theory of $G$ is said to be {\bf semi-simple} if every indecomposable
  $G$-module is irreducible.  
  
\begin{ex}
\begin{enumerate}
\label{ex:indecomposable}
\item
Let $G$ be a diagonalizable affine group scheme as in Definition \ref{defn:character}.
Then the representation theory of $G$ is semi-simple. 
\item
Let $G$ be the finite group $\bZ/p$; the coordinate algebra of $\bZ/p$ equals 
$Hom_{sets}(\bZ/p, k)$ whose dual algebra is the group algebra $k\bZ/p = k[x]/(x^p-1) \simeq k[t]/t^p$.  
There are $p$ distinct isomorphism classes of indecomposable $\bZ/p$-modules, represented (as
modules for $k[t]/t^p$) by the quotients $k[t]/t^i, \ 1 \leq i \leq p$ of $k[t]/t^p$.  
Only the 1-dimensional ``trivial" $kG$-module $k$ is irreducible.
\item
Let $G = GL_{n(1)}$, so that $kG \simeq U^{[p]}(\gl_n)$ and let $V = S^p(k^n) \simeq k[x_1,\ldots,x_n]_p$ 
denote the $p$-fold symmetric power of the defining representation $k^n$ of $GL_n$ 
(see Example \ref{ex:defining}).  Consider the subspace $W \ \subset V$ spanned by
$\{x_1^p,\ldots,x_n^p \}$.  Then $W$ is a $G$-submodule of $V$, but there does not exist another
$G$-submodule $V^\prime \subset V$ such that $V \simeq W \oplus V^\prime$.
\end{enumerate}
\end{ex}

In some sense, the ``ultimate goal" of the representation theory of $G$ is the description of all
isomorphism classes of indecomposable $G$-modules (as for $G = \Z/p$ in Example \ref{ex:indecomposable}.2.)
However, this goal is far too optimistic.  Even for $G = \bZ/p^{\times r}$ (for $r \geq 3$; for $p > 2$, we need
only that $r \geq 2$),
the representation theory of $G$ is ``wild", a condition which can be formulated as the condition that the  
abelian category of the finite dimensional representations of {\it any} finite dimensional $k$-algebra $\Lambda$ can 
be embedded in the abelian category $mod_k(G)$ of finite dimensional $G$-modules. (See, for example, \cite{BD}.)

\vskip .2in

\subsection{Derived functors}
We assume that the reader is familiar with the basics of homological algebra.  We refer the reader to C. Weibel's
book ``An Introduction to Homological Algebra" \cite{Wei} for background.  In our context, the
role of cohomology is to give information
about the structure of indecomposable $G$-modules, structure that arises by successive
extensions of irreducible $G$-modules. 

The following proposition (see \cite{J}) insures that the abelian category $Mod_k(G)$ has enough injectives,
thereby enabling the formulation of the $Ext_G^i(M,N)$ groups as right derived functors of the functor
$$Hom_G(M,-): Mod_k(G) \ \to \ (Ab)$$
from the abelian category of $G$-modules to the abelian category of abelian groups.  (Indeed, this functor 
takes values in the abelian category of $k$-vector spaces.)  As mentioned in the introduction, for ``most"
linear algebraic groups $G$, $Mod_k(G)$ has no non-trivial projectives \cite{Donk} so that we can not
define $Ext_G^*(-,-)$-groups by using a projective resolution of the contravariant variable.

\begin{prop}
\label{prop:enough}
Let $G$ be an affine group scheme over $k$.  Then $k[G]$ (with $G$-action given as the left regular representation)
is an injective $G$-module.  Moreover, if $M$ is any  $G$-module, then 
$M$ admits a natural embedding $M \hookrightarrow M \otimes k[G]$ and $M \otimes k[G]$ is an injective
$G$-module.
\end{prop}

\begin{defn}
\label{defn:cohomology}
Let $G$ be an affine group scheme over $k$.   For any pair of $G$-modules $M, N$ and any $i \geq 0$, we define
$$Ext_G^i(M,N) \ \equiv \ (R^i(Hom_G(M,-)))(N),$$
the value of the $i$-th right derived functor of $Hom_G(M,-)$ applied to $N$.

In particular, one has the graded commutative algebra $H^*(G,k) \ \equiv \ Ext_G^*(k,k)$.  For $p= 2, \ H^*(G,k)$ is
commutative.  For $p > 2$, we consider the commutative subalgebra $H^\bu(G,k) \subset H^*(G,k)$ generated by 
cohomology classes of even degree is a commutative $k$-algebra; for $p = 2$, we set $H^\bu(G,k)$ equal to the
commutative $k$-algebra $H^*(G,k)$.
An important theorem of B. Venkov \cite{V} 
and L. Evens \cite{Ev} asserts that
$H^*(G,k)$ is finitely generated for any finite group $G$; this was generalized to arbitrary finite group schemes by 
A. Suslin and the author \cite{FS}.
\end{defn}

\begin{remark}
One can describe $Ext^n_G(M,N)$ as the abelian group of equivalence classes of $n$-extensions of $M$ by $N$
(cf. \cite[III.5]{MacL}), where the equivalence relation arises by writing an $n$-extension as a composition
of $1$-extensions and using pushing forward and pulling back of $1$-extensions.
\end{remark}

\vskip .2in

\subsection{The Quillen variety $|G|$ and the cohomological support variety $|G|_M$}

In what follows, if $A$ is a finitely generated commutative $k$-algebra (such as $H^\bu(G,k)$
with grading ignored), then we denote by $\Spec A$ the affine scheme whose set of points
is the set of prime ideals of $A$ equipped with the Zariski topology and whose structure 
sheaf $\cO_{\Spec A}$ is a sheaf of commutative $k$-algebras whose value on $\Spec A$
is $A$ itself.  For $A = H^\bu(G,k)$, we denote by $|G|$ the topological space underlying $\Spec H^\bu(G,k)$;
in other words, we ignore the structure sheaf  $\cO_{\Spec A}$  on $G$.

The Atiyah-Swan conjecture for a finite group $G$ states that the {\it growth} of a minimal
projective resolution of $k$ as a $G$-module should be one less than the largest rank of 
elementary $p$-subgroup $E \simeq (\bZ/p)^{\times r} \subset G$.   This growth can be
seen to equal the Krull dimension of $H^\bu(G,k)$.

Daniel Quillen proved this conjecture and much more by introducing geometry into the study of $H^*(G,k)$.
A simplified version of Quillen's main theorem is the following.  Following Quillen, we let 
$\cE(G)$ be the category of elementary 
abelian $p$-groups of $G$ whose $Hom$-sets $Hom_{\cE}(E,E^\prime)$ consist of group homomorphisms
$E \to E^\prime$ which can be written as a composition of an inclusion followed by 
conjugation by an element of $G$.

\begin{thm}
\label{thm:quillen}
Let $G$ be a finite group.   If $\zeta \in H^*(G,k)$ is not nilpotent, then there exists some
elementary abelian $p$-subgroup $E \simeq (\bZ/p)^{\times r} \subset G$ such that $\zeta$
restricted to $H^*(E,k)$ is non-zero.

Furthermore, the morphisms $\Spec H^\bu(E,k) \to \Spec H^\bu(G,k)$ are natural with 
respect to $E \in \cE(G)$ and determine a homeomorphism
$$\varinjlim_{E < G} |E| \ \stackrel{\sim}{\to} \ |G|.$$
\end{thm}

This is a fantastic theorem.  Before Quillen's work, we knew very little about computations
of group cohomology and this theorem applies to all finite groups.  However, it actually
does not compute any of the groups $H^i(G,k)$ for $i>0$.   For example, 
$H^i(GL_{2n}(\bF_{p^d}),k) = 0, 1 \leq i \leq f(n,d)$ with $\varinjlim_d f(n,d) = \infty$.  On
the other hand, Theorem \ref{thm:quillen} tells us that the 
Krull dimension of $H^\bu(GL_{2n}(\bF_{p^d}),k)$ equals $d\cdot n^2$, for this is 
the rank of the largest elementary abelian $p$-group inside $GL_{2n}(\bF_{p^d})$.

J. Alperin and L. Evens initiated in \cite{AE1} the study of the growth of projective resolutions for an arbitrary 
finite dimensional $kG$-module for a finite group $G$ (extending Quillen's theorem for the
trivial $k$-module $k$) .  This led Jon Carlson in \cite{Ca1} to introduce the following notion of
the support variety of a finite group.

\begin{defn}
\label{defn:cohsupp}
Let $G$ be a finite group and denote by $|G|$ the space (with the Zariski topology) underlying $\Spec H^\bu(G,k)$. 
For  any finite dimensional $kG$-module $M$, denote by $I(M)
\subset H^\bu(G,k)$ the ideal of those elements $\alpha$ such that $\alpha$ acts as 
0 on $Ext^*_G(M,M)$.  The cohomological support variety $|G|_M$ is the closed subset of $|G|$ defined as the
``zero locus" of $I(M)$.  In other words, 
$$|G|_M \ = \ \Spec H^\bu(G,k)/I(M) \ \subset \ |G|.$$
\end{defn}

We remark that the ideal $I(M)$ of Definition \ref{defn:cohsupp} is equal to the kernel of the natural
map of graded $k$-algebras $H^*(G,k) \ \to \ Ext_G^*(M,M)$ given in degree $n$ by tensoring an $n$-extension 
of $k$ by $k$ by $M$ to obtain an $n$-extension of $M$ by $M$.

The following theorem states two of Carlson's early results concerning support varieties, both of which
have subsequently been shown to generalize to all finite group schemes.  The second result is especially
important (as well as elegant).

\begin{thm} (J. Carlson, \cite{Ca2})
Let $G$ be a finite group. 
\begin{enumerate}
\item
If $M$ is a finite dimensional indecomposable $G$-module, then
the projectivization of $|G|_M$ is connected.
\item
Let $C \subset |G|$ be a (Zariski) closed, conical subvariety of $|G|$.  Once given a choice 
of generators for the ideal in $H^\bu(G,k)$ defining $C$, one can explicitly 
construct a finite dimensional $kG$-module $M_C$ such that \ $|G|_{M_C} \  \simeq \ C$.  
\end{enumerate}
\end{thm}

Following the development of the theory of support varieties for finite groups,
various mathematicians considered the generalization of the theory to other
``group-like" structures as mentioned in the introduction.  

Definition \ref{defn:cohsupp} can be repeated verbatim for an arbitrary finite group scheme.
More interesting, a ``representation theoretic model" for $|G|_M$ has been developed for
any finite group scheme.  This began with the model $\cN_p(\fg)_M$ of  $|G|_M$  in terms of the 
$p$-nilptent cone $\cN_p(\fg)$  for a finite dimensional $U^{[p]}(\fg)$-module $M$, where $\fg$ is  an
an arbitrary finite dimensional $p$-restricted Lie $\fg$. This was extended to the model $V(G_{(r)})_M$ 
in terms of infinitesimal 1-parameter subgroups of $G$ for any  infinitesimal group scheme
$G_{(r)}$ in the work of  A. Suslin, C. Bendel, and the author.  Finally, in the work of
the author and J. Pevtsova, a model $\Pi(G)$ isomorphic to $|G|_M$ was formulated in
terms of equivalence classes of $\pi$-points.    (See Definition \ref{defn:pi-point} in
the next lecture.)

These geometric models for $|G|_M$ play an important role in proving the following properties 
of support varieties for these various ``group-like" structures.

\begin{thm}
\label{thm:coh-supports}
Let $G$ be a finite group scheme over $k$, and let $M, N$ be finite dimensional $G$-modules.
\begin{enumerate}
\item
$|G|_M = 0$ if and only if $M$ is a projective $G$-module if and only if $M$ is an injective $G$-module.
\item
$|G|_{M\oplus N} \ = \ |G|_M \cup |G|_N$.
\item
$|G|_{M\otimes N} \ = \ |G|_M \cap |G|_N$.
\item
For any short exact sequence $0 \to M_1 \to M_2 \to M_3 \to 0$ and any permutation $\sigma$ of $\{ 1,2,3 \}$,
$|G|_{\sigma(1)} \ \subset \ |G|_{\sigma(2)} \cup |G|_{\sigma(3)}$.
\end{enumerate}
\end{thm}

The theory of support varieties has not only given information about the representation theory of $G$
but also has led to new classes of modules.  We mention the ``modules of constant Jordan type" 
introduced by J. Carlson, J. Pevtsova, and the author \cite{CFP} based on the ``well-definedness of
maximal Jordan type" established by J Pevtsova, A. Suslin, and the author \cite{FPS}.  We point out
the paper of J. Calrson, Z. Lin, and D. Nakano \cite{CLN} which gives an interesting relationship
between the cohomological support variety for $G(\bF_p)$ and for $G_{(1)}$ for  finite dimensional $G$-modules
with $G$ equal to some simple algebraic group.

\begin{remark}
If $G$ is a linear algebraic group, then we face the following daunting problems in adopting the 
techniques of cohomological support varieties to the representation theory of linear algebraic groups. 
\begin{itemize}
\item  If $G$ is a simple algebraic group, then $H^i(G,k)$ vanishes in positive dimensions by the vanishing
theorem of G. Kempf \cite{Kem}.
\item  On the other hand, if $U$ is a non-trivial linear algebraic group which is unipotent,
then $H^\bu(G,k)$ is not finitely generated.
\item  We are unaware of a result which can play the role of  Quillen's detection theorem stating that cohomology of
a finite group is detected modulo nilpotents on elementary subgroups (see Theorem \ref{thm:quillen}).
\end{itemize}
\end{remark}

In recent work, the author has explored unipotent algebraic groups with the view that, unlike simple algebraic
groups, these should have ``enough cohomology".   Unfortunately, this appear not to be the case even for 
the Heisenberg group $U_3 \subset GL_3$ of upper triangular elements.    
%
%
%
%
A tentative framework has been developed by the author
in which a cohomological support theory $M \mapsto V^{coh}(G)_M$ is formulated using
a continuous approximation of the rational 
cohomology $H^\bu(G,k)$ of a linear algebraic group $G$.   This naturally maps to the support theory
$M \mapsto V(G)_M$ discussed in Lecture IV provided that $G$ is of ``exponential type".  For $G = \bG_a$,
this map is an isomorphism for all finite dimensional $\bG_a$-modules.  However, even for the Heisenberg
group $U_3$, the two theories are quite different.  For example, the image of $V^{coh}(G)_M \to V(G)_M$
is contained in the $G$-invariants of $V(G)$.

\vskip 1in


\section{Lecture IV: Support varieties for linear algebraic groups}

In his final lecture, the author presented his construction $M \ \mapsto  \ V(G)_M$ of support varieties
for $G$-modules $M$, where $G$ is a linear algebraic group ``of exponential type".  The beginnings
of this theory can be found in \cite{F1} and some applications in \cite{F3}.   The theory
succeeds in that the support varieties defined here extend those for infinitesimal kernels,
have many of the expected properties (see Theorem \ref{thm:G-items}), and are formulated 
intrinsically for those linear
algebraic groups for which the theory applies.  One interesting aspect of this theory
is that it leads to new and apparently interesting classes of (infinite-dimensional) 
$G$-modules.  

 One failure of the theory we present is that there are $G$-modules
$M$ which are not injective but for which $V(G)_M = 0$.
We hope that this theory will be refined, perhaps using ``formal 1-parameter subgroups"
mentioned at the end of this lecture. 

As for the first three lectures, we begin by providing the outline of this fourth lecture
given to participants.

\vskip .1in
\noindent
IV.A  1-parameter subgroups.

        i.)    Group homomorphisms $\bG_a \to G$;
        
        ii.)   Examples of $\bG_a$ and $GL_N$;
        
        iii.)  Springer isomorphisms and groups of exponential type;
        
        iv.)   SFB for $G_{(r)}$.

\vskip .1in
\noindent
IV.B  Linear algebraic groups of exponential type.

        i.)    Definitions; Sobaje's theorem;
        
        ii.)   $p$-nilpotent operator $\alpha_{\ul B}$;
        
        iii.)  Jordan types.

\vskip .1in
\noindent
IV.C  Support varieties.

        i.)    $V(G), V(G)_M$;
        
        ii.)   Example of $\bG_a$;
        
        iii.)  Properties.

\vskip .1in
\noindent
IV.D   Special modules.

        i.)    Mock injective modules;
        
        ii.)   Mock trivial modules;
        
        iii.)  Modules for realization of subspaces of $V(G)$.

\vskip .1in
\noindent
IV.E   Some open problems.

        i.)    Formal 1-parameter subgroups and injectivity;
        
        ii.)   Finite generation of cohomology of sub-coalgebras;
        
        iii.)  Detecting rational cohomology modulo nilpotents.

\vskip .1in
\noindent
Topics for discussion/projects:

\vskip .1in
\noindent
IV.A   Work through the exponential map for $GL_N$; work through
       some details of proofs found in \cite{SFB1}, \cite{SFB2}.
\vskip .1in
\noindent
IV.B   Investigate 1-parameter groups for $Sp_{2n}$.
\vskip .1in
\noindent
IV.C   Work out examples for $\bG_a$, for induced modules, for homogeneous
            varieties.
\vskip .1in
\noindent
IV.D  Investigate the question of what $C \subset V(G)$ can be realized as 
$V(G)_M$ for some (possibly infinite dimensional) $G$-module $M$.

\vskip .2in

\subsection{1-parameter subgroups}

In this subsection, we discuss 1-parameter subgroups of linear algebraic groups.
 These 1-parameter subgroups  might more formally be called {\it unipotent} 1-parameter 
 subgroups.  After giving the definition and some examples, we give the definition
 of a linear algebraic group of exponential type.  For such a group $G$, the set
 $V(G)$ of 1-parameter subgroups is the set of $k$-points of an ind-scheme 
 $\cC_{\infty}(\cN_p(\fg))$ defined in terms of the restricted Lie algebra $\fg$ of $G$.
 As we mention, most of the familiar linear algebraic groups are groups of exponential type.
 
 We begin by recalling from \cite{SFB1} the affine scheme $V_r(G)$ of height $r$ infinitesimal 1-parameter
 subgroups of an affine group scheme $G$ over $k$.
 
 \begin{defn}
 \label{defn:1par-scheme}
 Let $G$ be an affine group scheme over $k$ and $r$ a positive integer.  Then the functor sending
 a commutative $k$-algebras $A$ to the set of morphisms (over $\Spec A$) of group schemes 
 of the form $\bG_{a(r),A} \to G_A$ is representable by an affine group scheme $V_r(G)$.
 Here, $G_A$ is the base change $G \times_{\Spec k} \Spec A$ of $G$.
 
 In particular, $V_r(G)(k)$ is the set of height $r$  infinitesimal 1-parameter subgroups 
 $\mu: \bG_{a(r)} \to G$.
 \end{defn}
 
For any affine group scheme $G$ over $k$,  $V_r(G) = V_r(G_{(r)})$.

\begin{defn}
\label{defn:1par}
Let $G$ be a linear algebraic group over $k$.  Then a 1-parameter subgroup
is a morphism of group schemes over $k$ of the form $\psi: \bG_a \to G$.  We 
denote by $V(G)$ the set of 1-parameter subgroups of $G$.  

Restriction to $G_{(r)}$ determines a natural map $V(G) \ \to (V_r(G))(k) = (V_r(G_{(r)})(k)$
from $V(G)$ to the set of infinitesimal 1-parameter subgroups $\bG_{a(r)} \to G$,
the set of  $k$-points of the affine scheme $V(G_{(r)})$ of Definition \ref{defn:1par-scheme}.
\end{defn}

\begin{ex}
\begin{enumerate}
\item
Take $G = \bG_a$.  A 1-parameter subgroup $\bG_a \to \bG_a$ is determined by a map of 
coordinate algebras $k[T] \ \leftarrow \ k[T]$ given by sending $T$ to an additive polynomial; namely
a polynomial of the form $ \sum_{i \geq 0} a_iT^{p^i}$.  (The condition that the map 
$k[T] \ \leftarrow \ k[T]$ sending $T$ to $p(T)$ is a map of
Hopf algebras is equivalent to the condition that $p(T)$ be of this form.)  Thus, $V(\bG_a)$ is the set of $k$-points
of the affine ind-scheme $\bA^\infty$, the set of all sequences $\ul a = (a_0,a_1,\ldots,a_n,\dots )$
with the property that $a_N = 0$ for $N$ sufficiently large (i.e., ``finite sequences").
\item
Take $G = GL_N$.  Then a 1-parameter subgroup $\psi: \bG_a \to GL_N$ has associated map
on coordinate algebras $k[\{ X_{i,j}\},det^{-1}] \to k[T]$ which must be compatible with the
coproducts $\Delta_{GL_N}$ and $\Delta_{\bG_a}$.   As shown in \cite {SFB1}, such a 1-parameter subgroup 
corresponds to a a finite sequence $\ul A = (A_0,A_1,\ldots,A_N, \ldots)$ of $p$-nilpotent $N\times n$ matrices
(i.e., $p$-nilpotent elements of $\gl_N$) which pair-wise commute.  To such a finite sequence $\ul A$, the associated
1-parameter subgroup is the morphism of algebraic groups
$$\prod_{i \geq 0} exp_{A_i} \circ F^i: \bG_a \to \bG_a \to GL_N, \quad r \in R \mapsto \prod_{i\geq 0} exp_{A_i}(r^{p^i})
\in GL_N(R),$$
where 
$$exp_A(s) = \ 1 + s\cdot A + (s^2/2) \cdot A^2 + \cdots + (s^{p-1}/(p-1)! )\cdot A^{p-1}.$$
Thus, $V(GL_N)$ is the set of affine $k$ points of the ind-scheme $\cC_\infty(\cN_p(\gl_N)) \ = \ 
\varinjlim_r \cC_r(\cN_p(\gl_N))$, where $\cC_r(\cN_p(\gl_N)) \simeq V_r(GL_N)$ represents the functor of $r$-tuples
of $p$-nilpotent, pair-wise commuting $N\times N$ matrices.
\end{enumerate}
\end{ex}
\vskip .1in

\begin{defn}
Let $\fg$ be a finite dimensional restricted Lie algebra over $k$.  Denote by $\cN_p(\fg)$
the subvariety of $\fg$ (viewed as an affine space) consisting of $X\in \fg$ with $X^{[p]} = 0$.
We define the affine $k$-scheme $\cC_r(\cN_p(\fg))$ to be the subvariety of $(\cN_p(\fg))^{\times r}$
consisting of $r$-tuples $(B_0,\ldots,B_r)$ satisfying 
$$[B_i,B_j] \ = \ B_i^{[p]} \ = \ B_j^{[p]} \ = \ 0, \ 0 \leq i,j \leq r.$$
We define $\cC_\infty(\cN_p(\fg))$ to be the ind-scheme  $\varinjlim \cC_r(\cN_p(\fg))$.
\end{defn}

Our construction of support varieties only applies to a linear algebraic group $G$ which
is of {\bf exponential type}.   This condition is the condition that $V(G)$ can be 
naturally identified with the 
set of $k$ points of $\cC_\infty(\cN_p(\fg))$ as is the case for $G = GL_N$.
The following definition of \cite{F2} is an extension of the concept in \cite{SFB1} 
of an embedding $G \subset GL_N$ of exponential type.

\begin{defn}
\label{defn:exptype}
Let $G$ be a linear algebraic group over $k$ with Lie algebra $\fg$.   A structure of exponential type
on $G$ is a morphism of $k$-schemes
\begin{equation}
\label{Exp}
\cE: \cN_p(\fg) \times \bG_a \ \to G, \quad (B,s) \mapsto \cE_B(s)
\end{equation}
such that
\begin{enumerate}
\item
For each $B\in \cN_p(\fg)(k)$, $\cE_B: \bG_a \to G$ is a 1-parameter subgroup.
\item
For any pair of  commuting $p$-nilpotent elements $B, B^\prime \in \fg$,
the maps $\cE_B, \cE_{B^\prime}: \bG_a \to G$ commute.
\item
For any commutative $k$-algebra $A$, any $\alpha \in A$,  and any 
$s \in \bG_a(A)$, \ $\cE_{\alpha \cdot B}(s) = \cE_B(\alpha\cdot s)$.
\item  Every 1-parameter subgroup $\psi: \bG_a \to G$ is of the form 
$$ \cE_{\ul B} \ \equiv \ \prod_{s=0}^{r-1} (\cE_{B_s} \circ F^s)$$
for some $r > 0$, some $\ul B \in \cC_r(\cN_p(\fg))$; furthermore, $\cC_r(\cN_p(\fg)) \to
V_r(G), \ \ul B \mapsto \cE_{\ul B} \circ i_r$ is an isomorphism for each $r > 0$.
\end{enumerate}

A linear algebraic group over $k$ which admits a structure of exponential type is
said to be a {\bf linear algebraic group of exponential type}.

Moreover,  a closed subgroup $H \subset G$ is said to be an embedding of {\it exponential type} if
$H$ is equipped with the structure of exponential type given by restricting that
provided to $G$; in particular, we require $\cE: \cN_p(\fg) \times \bG_a \ \to \ G$
to restrict to $\cE: \cN_p(\fh) \times \bG_a \ \to \ H$.
\end{defn}

Up to isomorphism, if such a structure exists then it is unique.

\begin{ex} There are many examples of linear algebraic groups of exponential type.
\begin{enumerate}
\item
Any classical simple linear algebraic group $G$ over $k$ (i.e., of type $A, B, C$ or $D$) and 
the unipotent radical of any parabolic subgroup defined of such a group $G$ over $\bF_p$ as remarked in 
\cite{SFB1}.
\item
Any simple linear algebraic group $G$ provided that $p$ is separably good for $G$ (see \cite{S3}).
\item
Any term of the lower central series of the unipotent radical of a parabolic subgroup defined
over $\bF_p$ of a simple algebraic group $G$, provided $p$ is separably good for $G$ (see \cite{Sei} plus \cite{S3}).
\end{enumerate}
\end{ex}

\subsection{$p$-nilpotent operators}

We begin this subsection by briefly recalling the theory of $\pi$-points for finite group schemes
developed by J. Pevtsova and the author.

If $G$ is a linear algebraic group over $k$ and $M$ a rational $G_{(r)}$-module, then the 
geometric formulation $V_r(G_{(r)})_M$ of $|G_{(r)}|_M$ is obtained by associating to every point of $V_r(M)$ a 
$p$-nilpotent operator on $M$.   In the following definition, we use field extensions $K/k$ to capture
the scheme structure of $V_r(G_{(r)})_M$.

\begin{defn}
\label{defn:pi-point}
Let $G$ be a finite group scheme with group algebra $kG$ (the $k$-linear dual to the coordinate 
algebra $k[G]$).  Then a $\pi$-point is a left flat $K$-linear map of algebras $\alpha_K: K[t]/T^p \to KG$ 
for some field extension $K/k$
with the property that $\alpha_K$ factors through $KC_K \to KG$ for some unipotent subgroup scheme
$C_K \subset G_K$. 

For a suitable equivalence relation on $\pi$-points, the set of equivalence classes of $\pi$-points of $G$ is 
naturally identified with the set of non-tautological homogeneous prime ideals of $H^\bu(G,k)$.  Indeed, 
one can put a scheme structure $\Pi(G)$ on equivalence
classes of $\pi$-points which is
formulated in terms of the category of  $G$-modules (and not using homological algebra) so that
$\Pi(G)$ is isomorphic to $\Proj H^\bu(G,k)$ as a $k$-scheme \cite{FPv2}.

For any  $G$-module $M$, the ``local action" on $M$ at the $\pi$-point $\alpha: K[t]/T^p \to KG$ 
is the action of $\alpha_{K*}(T)$ on $M_K \equiv M\otimes K$ (equivalently, the action of $T \in k[T]/T^p$ on $\alpha_K^*(M_K)$).

For any $G$-module $M$, the ``$\Pi$-support variety" $\Pi(G)_M$ of $M$ consists of those 
equivalence classes of $\pi$-points $\alpha: K[t]/T^p \to KG$ for which $\alpha^*(M_K)$ is {\it not} free
as a $K[T]/T^p$-module.
\end{defn}

The fact that $\Pi(G)_M$ is well defined (that the condition that an equivalence of class of $\pi$-points
can be tested on any representative of that equivalence class) was justified by the work of J. Pevtsova,
and the author in \cite{FPv1}.

The following definition of ``local action" of $G_{(r)}$ on $M$ at an infinitesimal 1-parameter
subgroup is implicit in \cite{SFB2}.

\begin{defn}
\label{defn:local-Frob}
Let $G$ be a linear algebraic group of exponential type, let $\ul B = (B_0,\ldots,B_{r-1})$ be a $k$-point of
$\cC_r(\cN_p(\fg))$, and let $M$ a $G_{(r)}$-module.   Then the local action of $G_{(r)}$
on $M$ at $\cE_{\ul B}$ is defined to be the local action at the $\pi$-point 
$\mu_{\ul B} \equiv \cE_{\ul B*} \circ \epsilon_r: k[T] \to k\bG_{a(r)} \to kG_{(r)}$
sending $T$ to $\cE_{\ul B*}(u_{r-1})$.   (The map $\epsilon_r: k[T]/T^p \to k\bG_{a(r)} = 
k[u_0,\ldots,u_{r-1}]/(\{ u_i^p \})$ is the map of $k$-algebras sending $T$ to $u_{r-1}$; 
this is a Hopf algebra map if and only if $r=1$.) 

Consequently, 
$$V_r(G)_M = V(G_{(r)}) \ \simeq \ \{ \cE_{\ul B} \in V_r(G): \mu_{\ul B}^*(M) \text{\ is not free} \}.$$
\end{defn}

After much experimentation, the author introduced in \cite{F1} the following
definition of the local action at a 1-parameter subgroup $\cE_{\ul B}$ of a linear algebraic group of 
exponential type $G$ acting on a $G$-module $M$.  
This definition is not formulated in terms of  $\cE_{\ul B}^*(M)$.  The justification of 
the somewhat confusing ``twist" (i.e., a reordering of $\ul B = (B_0,\ldots,B_r\ldots)$ is 
implicit in Proposition \ref{prop:comparepi}, which shows that the restriction to Frobenius kernels
of this definition gives a ``functionally equivalent" formulation of ``local action" as that given in 
Definition \ref{defn:local-Frob}.

\begin{defn}
\label{defn:localaction}
Let $G$ be a linear algebraic group of exponential type, equipped with an exponentiation
$\cE: \cN_p \times \bG_a \to G$.  Let $M$ be a rational $G$-module and $\ul B = (B_0,B_1,\ldots,B_n,\ldots)
\in \cC_\infty(\cN_p(\fg))$ be a finite sequence.  Then the action of $G$ on $M$ at  
$\cE_{\ul B}: \bG_a \to G \in V(G)$ is defined to be the action of
\begin{equation}
\label{eqn:localaction}
\sum_{s \geq 0} (\cE_{B_s})_*(u_s) \ = \ \sum_{s \geq 0} (\cE_{B_s} \circ F_s)_*(u_0).
\end{equation}
\end{defn}

One checks that this action is in fact $p$-nilpotent, thereby defining 
\begin{equation}
\label{eqn:alphaB}
\alpha_{\ul B}: k[u]/u^p \to \ kG, \quad \ul B \in \cC_\infty(\cN_p(\fg))); \quad u \mapsto 
\sum_{s \geq 0} (\cE_{B_s})_*(u_s).
\end{equation}

The close connection of Definition \ref{defn:localaction} and the theory of $\pi$-points
briefly summarized in Definition \ref{defn:pi-point} is given by the following result of \cite{F2}
based upon an argument of P. Sobaje \cite{S1}. 

\begin{prop} \cite[4.3]{F2}
\label{prop:comparepi}
Let $G$ be a linear algebraic group of exponential type, equipped with an exponentiation
$\cE: \cN_p \times \bG_a \to G$.  For any $r > 0$ and any $\ul B \in \cC_r(\cN_p(\fg))$, 
the $\pi$-points of $G_{(r)}$
$$\mu_{\ul B} = \cE_{\ul B} \circ \epsilon_r : k[T]/T^p \ \to k\bG_{a(r)} \ \to\ kG_{(r)}, \quad 
\alpha_{\Lambda_r(\ul B)}: k[u]/u^p \to kG_{(r)}$$
are equivalent, where $\Lambda_r(B_0,\ldots,B_{r-1}) = (B_{r-1},\ldots,B_0)$.
\end{prop}

This equivalence of $\pi$-points enables a comparison of support varieties for finite
group schemes and the definition we now give of support varieties for linear algebraic groups 
of exponential type.  Indeed, it enables a comparison of the ``generalized support varieties"
introduced by J. Pevtsova and the author in \cite{FPv3} using the local data of the full
Jordan type of a $k[u]/u^p$-module rather than merely whether or not such a module is free.

\vskip .2in

\subsection{ The support variety $V(G)_M$}

Much of this subsection is copied from the author's paper \cite{F2}.
After giving the definition of the support variety $V(G)_M$ of a rational $G$-module $M$ of a linear 
algebraic group of exponential type, we review many of the properties of this construction.
The first property of Theorem \ref{thm:G-items} tells us that $V(G)_M$ can be
recovered from $V(G_{(r)})_M$ for $r >> 0$ provided that $M$ is finite dimensional.
On the other hand, for $M$ infinite dimensional the support variety $V(G)_M$ provides 
information about $M$ not detected by any Frobenius kernel.

\begin{defn}
\label{supportG}
Let $G$ be a linear algebraic group equipped with a structure of  exponential type and let
$M$ be a  $G$-module. 
We define the {\bf support variety} of $M$ to be the subset $V(G)_M \ \subset \ V(G)$ consisting 
of those $\cE_{\ul B}$ such that $\alpha_{\ul B}^*(M)$  is not free as a $k[u]/u^p$-module, where 
$\alpha_{\ul B}: k[u]/u^p \to kG$ is defined in (\ref{eqn:alphaB}).

For a finite dimensional $G$-module $M$, we define the {\bf Jordan type} of $M$ at the 
1-parameter subgroup $\cE_{\ul B}$ to be
$$JT_{G,M}(\cE_{\ul B})\quad \equiv \quad JT(\sum_{s \geq 0} (\cE_{B_s})_*(u_s),M),$$
the Jordan type of the local action of $G$ on $M$ at $\cE_{\ul B}$ (see Definition \ref{defn:localaction}).
For such a  finite dimensional $G$-module $M$, 
$V(G)_M \ \subset \ V(G)$ consists of those 1-parameter subgroups $\cE_{\ul B}$ 
such that some block of the Jordan type of $M$ at $\cE_{\ul B}$ has size $< p$.
\end{defn}

The following definition is  closely
related to the formulation of $p$-nilpotent degree given in \cite[2.6]{F1}.

\begin{defn} (cf. \cite[2.6]{F1})
\label{expdegree}
Let $G$ be a linear algebraic group equipped with a structure of exponential type and let $M$ be a 
$G$-module.  Then $M$ is said to have  exponential degree $< p^r$
if $(\cE_B)_*(u_s)$ acts trivially on $M$ for all $s \geq r$, all $B \in \cN_p(\fg)$.
\end{defn}

As observed in \cite{F2}, every finite dimensional $G$-module $M$ has
exponential degree $< p^r$ for $r$ sufficiently large.

\begin{thm} \cite[4.6]{F2}
\label{thm:G-items}
Let $G$ be a linear algebraic group equipped with a structure of  exponential type and
$M$ a rational $G$-module
\begin{enumerate}
\item  
If $M$ has exponential degree $< p^r$, then 
$V(G)_M \ = \ \Lambda_r^{-1}(V_r(G)_M(k)))$. 
\item
If $M$ is finite dimensional, then $V(G)_M \subset V(G)$ is closed.
\item
$V(G)_{M\oplus N} = V(G)_M \cup V(G)_N.$
\item
$V(G)_{M\otimes N} = V(G)_M \cap V(G)_N.$
\item
If $0 \to M_1 \to M_2 \to M_3 \to 0$ is a short exact sequence
of rational $G$-modules, then the support variety $V(G)_{M_i}$ of one of the
$M_i$'s is contained in the union of the support varieties of the other two.
\item 
If $G$ admits an embedding $i: G \hookrightarrow GL_N$ of exponential type, then 
$$V(G)_{M^{(1)}} \ = \ \{ \cE_{(B_0,B_1,B_2\ldots)} \in V(G) :
 \cE_{(B_1^{(1)},B_2^{(1)},\ldots)} \in V(G)_M\} .$$
 (Here, $M^{(1)}$ is the Frobenius twist of $M$, as formulated in Definition \ref{Frobenius twist}.)
\item  For any $r >0$, the restriction of $M$ to $kG_{(r)}$ is injective (equivalently, projective) 
if and only if the intersection of $V(G)_M$ with the subset 
$\{ \psi_{\ul B}: B_s = 0, s > r \}$ inside $V(G)$ equals $\{ \cE_{\ul 0} \}$.
\item  $V(G)_M \ \subset \ V(G)$ is a $G(k)$-stable subset.
\end{enumerate}
\end{thm}

\begin{remark}
\label{rem:tw}
In \cite[4.6]{F2}, property (6) was proved under the assumption that  $i: G \hookrightarrow GL_N$
be defined over $\bF_p$.  This is unnecessary, for \cite[1.11]{F2} also does not require $i$ to 
be defined over $\bF_p$.  Namely, one uses the diagram
\begin{equation}
\label{eqn:relateF}
\begin{xy}*!C\xybox{%
\xymatrix{
\bG_a  \ar[r]^{\psi_{\ul B}} & G \ar[d]^{i} \ar[r]^{F} & G^{(1)} \ar[d]^{i^{(1)}} \\
 & GL_N \ar[r]^-{F} & GL_N^{(1)} = GL_N   }
}\end{xy}
\end{equation}
to reduce to verifying both \cite[1.11]{F2} and \cite[4.6]{F2} in the special case $G = GL_N$.
\end{remark}

We repeat a remark which suggests how to formulate support varieties in the category of strict 
polynomial functors.

\begin{remark} \cite[4.7]{F2}
\label{rem-poly}
A special case of Theorem \ref{thm:G-items} is the case $G = GL_n$ and $M$ a polynomial $GL_n$-module
homogenous of some degree as in Example \ref{ex:groupalg}.5.  In particular, Theorem \ref{thm:G-items} 
provides a theory of support varieties for modules over the Schur algebra $S(n,d)$ for $n \geq d$.
\end{remark}

A few examples of such support varieties are given in \cite{F2}.  We give another interesting
example here, an infinite dimensional version of J. Carlson's $L_\zeta$-modules.

\vskip .1in

We consider a class of examples $Q_\zeta$ 
associated to  rational cohomology classes $\zeta \in H^\bu(G,k)$.  
Given a linear algebraic group $G$ and some choice of  injective resolution
$k \to I^0 \to I^1 \to  \cdots \to I^n \to \cdots$ of rational $G$-modules, we set 
$\Omega^{-2d}(k) $ equal to the quotient of $I^{2d-1}$ modulo the image of $I^{2d-2}$.   
The restriction of $\Omega^{-2d}(k)$ to some Frobenius 
kernel $G_{(r)}$ of $G$ is equivalent in the stable category of $G_{(r)}$-modules to 
$\Omega_{(r)}^{-2d}(k)$ defined as the quotient of  $I_{(r)}^{2d-1}$ modulo the image of $I_{(r)}^{2d-2}$ for a
minimal injective resolution $k \to I_{(r)}^0 \to \cdots \to I_{(r)}^n \to \cdots$ of rational $G_{(r)}$-modules.

\begin{prop}
\label{prop:qzeta}
Let $G$ be a linear algebraic group equipped with a structure of exponential type.
Consider a  rational cohomology 
class $\zeta \in H^{2d}(G,k)$ represented by a map $\tilde \zeta: k \to \Omega^{-2d}(k)$
of rational $G$-modules.  We define $Q_\zeta$ to be the cokernel of $\tilde \zeta$, 
thus fitting in the short exact sequence of rational $G$-modules
\begin{equation}
\label{seqQ}
0 \to k \stackrel{\widetilde \zeta}{\to} \ \Omega^{-2d}(k) \ \to Q_\zeta \to 0.
\end{equation}
Then 
$$V(G)_{Q_\zeta} \quad = \quad  \bigcup_r \{ \cE_{\ul B} \in V_r(G): 
(\alpha_{\Lambda_r(\ul B)})^*(\zeta) = 0 \in H^{2d}(k[t]/t^p,k)  \}.$$
\end{prop}

\begin{proof}
Since $\Omega_{(r)}^{-2d}(k)$ is stably equivalent as a $G_{(r)}$-module to the restriction to
$G_{(r)}$ of the rational $G$-module $\Omega^{-2d}(k)$, we conclude that the restriction of 
$Q_\zeta$ to $G_{(r)}$ is stably equivalent to the finite dimensional $G_{(r)}$-module
$Q_{\zeta_r}$ (associated to $\zeta_r \in H^{2d}(G_{(r)},k)$, the restriction of $\zeta$)
fitting in the short exact sequence 
$$0 \to k \to \ \Omega_{(r)}^{-2d}(k) \ \to Q_{\zeta_r} \to 0.$$
By definition, the Carlson $L_{\zeta_r}$-module introduced in \cite{Ca1} fits in the distinguished triangle
$$ L_{\zeta_r} \ \to \ \Omega^{2d}_{(r)}(k) \ \to \ k \to\  \Omega_{(r)}^{-1}(L_\zeta) $$
whose $[-2d]$-shift is  the distinguished triangle 
$$ \Omega^{-2d}_{(r)}(L_{\zeta_r}) \ \to \ k \ \to \Omega^{-2d}_{(r)}(k) \ \to \ \Omega^{-2d-1}_{(r)}(L_{\zeta_r}). $$
Thus, we conclude that  $Q_{\zeta_r}$ is stably equivalent to $\Omega^{-2d-1}_{(r)}(L_{\zeta_r}),$
and thus has the same support as a $G_{(r)}$-module as the support of $L_{\zeta_r}$. 

Observe that $V(G)_{Q_{\zeta}} \cap V_rG)$ equals $V(G_{(r)})_{Q_{\zeta_r}}$.
The identification of $V(G)_{L_\zeta}$ now follows from  \cite[3.7]{FPv2} which asserts that 
$$V(G_{(r)})_{L_{\zeta_r}} \ = \ \{ \mu: \bG_{a(r)} \to G_{(r)}, (\mu \circ \epsilon_r)^*(\zeta_r) = 0 \}.$$ 
\end{proof}

\vskip .2in

\subsection{Classes of rational $G$-modules}

The $\pi$-point approach to support varieties for a
finite group scheme $G$ naturally led to the formulation of the class of $G$-modules
of constant Jordan type (and more general modules of constant $j$-type for $1 \leq j < r$).   
For an infinitesimal group scheme $G$, 
J. Pevtsova and the author in \cite{FPv4} showed how to construct various vector
bundles on $V(G)$ associated to a $G$-module $M$ of constant Jordan type.
 We refer the reader to the book 
by D. Benson \cite{Ben} and the paper  by D. Benson and J. Pevtsova \cite{BenPv}
for an exploration of vector bundles constructed in this manner 
for elementary abelian $p$-groups $E \simeq (\bZ/p)^{\times s}$.  This is a ``special case"
of an infinitesimal group scheme because the representation theory of $(\bZ/p)^{\times s}$ is that of 
the height 1 infinitesimal group scheme $\bG_{a(1)}^{\times s}$.
We also mention that J. Carlson, J. Pevtsova, and the author introduced in \cite{CFP2} a 
construction which produced vector bundles on Grassmann varieties associated to 
modules of constant Jordan type as well as to more general modules, those of constant $j$-type.

In this subsection, we briefly mention three interesting classes of (infinite dimensional) 
$G$-modules for $G$ a linear algebraic group of exponential type.  Consideration of
special classes of $G$-modules is one means of obtaining partial understanding of
the wild category $Mod_k(G)$.

Throughout this subsection $G$ will denote a linear algebraic group of exponential type.

\begin{defn}
\label{mockinj}
We say that $M$ is mock injective if $M$ is not injective but $V(G)_M = 0$.  
\end{defn}

As the author showed in \cite{F3}, using results of E. Cline, B. Parshall, and L. Scott \cite{CPS1} on
the relationship of induced modules (see Definition \ref{defn:induced}) to injectivity, such mock 
injectives exist for any unipotent algebraic group which is of exponential type.
Necessary and sufficient conditions on $G$ for the existence of mock
injectives can be found in \cite{HNS}, once again using induction.

\begin{defn}
\label{mocktrivial}
We say that $M$ is mock trivial if the local action of $G$ on $M$ is trivial for all
1-parameter subgroups $\cE_{\ul B}: \bG_a \to G$.  
\end{defn}

In \cite{F3}, the author shows how to construct mock trivial $G$-modules for any
$G$ which is not unipotent.

\begin{defn}
\label{mockexpdeg}
We say that $M$ is of mock exponential degree $< p^r$ if there exists some $r > 0$ such
that $V(G)_M = \Lambda_r^{-1}(V(G_{(r)})_M)$.
\end{defn}

This class of $G$-modules includes all finite dimensional $G$-modules.

\vskip .2in

\subsection{Some questions of possible interest}

In this final subsection, we mention some questions which might interest the reader,
some of which concern the special classes of  $G$-modules
defined in the previous subsection.

\begin{question}  
For certain linear algebraic groups $G$ (e.g., $\bG_a$), can one describe the monoid
(under tensor product) of mock injective $G$-modules with 1-dimensional socle?
\end{question}

\begin{question}
Can one characterize $G$-modules of bounded mock exponential degree 
using $G$-modules which are extensions of  mock injective modules by finite dimensional modules?
\end{question}

\begin{question}
What conditions on a subset $X \subset V(G)$ imply good properties of the subcategory of 
$Mod_k(G)$ consisting of those $G$-modules $M$ with $V(G)_M \subset X$ for some $G$-module $M$? 
\end{question} 

\begin{question}
What are (necessary and/or sufficient) conditions on a subset $X \subset V(G)$ 
to be of the form $V(G)_M$? 
\end{question} 

\begin{question}
Do there exist rational cohomology classes $\zeta \in H^d(G,k)$ for some $G$ and
some $d > 0$ which are not nilpotent but which satisfy the condition that 
$\cE_{\ul B}^*(\zeta) = 0 \in H^*(k[T]/T^p,k)$ for all $\cE_{\ul B} \in V(G)$?
\end{question} 

\begin{question}
As in \cite{F3}, we have natural filtrations of $k[G]$ by subcoalgebras \\
$C \subset k[G]$.   Especially for
$G$ unipotent, can we prove finiteness theorems for \\
$Ext_{C-coMod}^*(M,M)$ for $M$ a finite dimensional rational $G$-module which inform
questions about $Ext^*_G(M,M)$?
\end{question}
\vskip .2in

We conclude with a possible ``improvement" of our support theory $M \ \mapsto \ V(G)_M$
for linear algebraic groups of exponential type.
The formulation of $V(G)$ presented in this text  (and in 
\cite{F2}) is that of a colimit \ $\varinjlim_r V_r(G)$, \ where $V_r(G) \simeq \cC_r(\cN_p(\fg))$. 

What follows is an alternative support theory, $M \ \mapsto \ \widehat{V(G)}_M$.  We remind
the reader of the affine scheme $V_r(G)$ given in Definition \ref{defn:1par-scheme} for any
affine group scheme $G$: the set of $A$-points of $V_r(G)$ is the set of the morphisms of group 
schemes  $\bG_{a(r),A} \to G_A$ over $Spec A$.

\begin{defn}
\label{defn:pro-V}
Let $G$ be a linear algebraic group.  For each $r>0$, we define the restriction morphism
$V_{r+1}(G) \to V_r(G)$ by restricting the domain of a height $r+1$ \\ 1-parameter subgroup
$\bG_{a(r+1),A} \to G_A$ to $\bG_{a(r),A} \subset \bG_{a(r+1),A}$.  

Thus, $\{ V_r(G), r > 0 \}$ is a pro-object of affine schemes.
\end{defn}

If $G$ is a linear algebraic group of exponential type, the restriction map $V_{r+1}(G) \to V_r(G)$ 
is given by  the projection $\cC_{r+1}(\cN_p(\fg)) \to  \cC_r(\cN_p(\fg))$ onto the first $r$ factors.   

\begin{defn}
\label{defn:formal}
For $G$ a linear algebraic group of exponential type, we define 
$$\widehat{V(G)} \ = \varprojlim_r \{ V(G_{(r)})(k) \}$$
equipped with the topology of the inverse limit of the Zariski topologies on the 
sets of $k$-points $V(G_{(r)})(k)$.  

We view an element of $\widehat{V(G)}$ as a ``formal 1-parameter subgroup"
 given as an infinite product 
$\widehat{\cE_{\ul B}}  = \prod_{s =0}^\infty \cE_{B_s}\circ F^s$.
\end{defn}

One proves the following proposition by using the following observation: for any
coaction $\Delta_M: M \to M \otimes k[G]$ and any $m \in M$, there exists a positive integer
$s(m)$ such that the composition 
$$u_s \circ \cE_B^* \circ \Delta_M: M \to M\otimes k[G] \to M\otimes k[T] \to k$$
vanishes on $m$ for all $B \in \cN_p(\fg)$ and all $s \geq s(m)$.

\begin{prop}
\label{prop:formal}
Let $G$ a linear algebraic group of exponential type and $M$ a $G$-module.   
Then for any $\widehat{\cE_{\ul B}} \in \widehat{V(G)}$ and any $m \in M$, 
the infinite sum $\sum_{s = 0}^\infty (\cE_{B_s})_*(u_s)$ applied to $m$ is finite
(i.e. $ (\cE_{B_s})_*(u_s)$ applied to $m$ vanishes for $s >> 0$).

Consequently, $\sum_{s = 0}^\infty (\cE_{B_s})_*(u_s)$ defines a $p$-nilpotent 
operator $\psi_{\ul B,M}: M \to M$.  We define
$$\widehat{V(G)}_M \ \equiv \ \{ \widehat{\cE_{\ul B}}: \  \text{not \ all blocks of } \psi_{\ul B,M} \
{\text have \ size} \ p \}.$$
\end{prop} 

We conclude with the following questions concerning \ $M \mapsto \widehat V(G)_M$. 

\begin{question}
Does use of formal 1-parameter subgroups provide necessary and sufficient conditions for
injectivity of a $G$-module?

Does use of formal 1-parameter subgroups provide necessary and sufficient conditions for
a $G$-module to be of bounded exponential degree?
\end{question}

\vskip 1in


\end{document}